\documentclass[final,leqno,onefignum,onetabnum]{siamltex1213}

\usepackage{color,epsfig,amssymb,amsmath,subfigure,url}
\usepackage{amsfonts}
\usepackage{multirow}
\usepackage{epstopdf}
\usepackage{latexsym}
\usepackage{enumerate}
\usepackage{tabularx}
\usepackage{booktabs}

\newtheorem{thm}{Theorem}[section]

\newtheorem{thm4}{Theorem}[section]

\newtheorem{lem}[thm]{Lemma}
\newtheorem{prop}[thm]{Proposition}
\newtheorem{defn}[thm]{Definition}
\numberwithin{equation}{section}

\newtheorem{example}[thm4]{Example}

\def \spn {\,{\rm span}\,}

\newcommand{\se}{\text{e}}

\newcommand{\bN}{\mathbb{N}}

\newcommand{\va}{\boldsymbol{a}}
\newcommand{\vb}{\boldsymbol{b}}
\newcommand{\vc}{\boldsymbol{c}}
\newcommand{\vx}{\boldsymbol{x}}

\def \vj {{\boldsymbol{j}}}
\def \vl {{\boldsymbol{l}}}
\def \vm {{\boldsymbol{m}}}
\def \vn {{\boldsymbol{n}}}

\def \vlambda {{\boldsymbol{\lambda}}}
\def \vepsilon {{\boldsymbol{\epsilon}}}

\newcommand{\vf}{\boldsymbol{f}}

\newcommand{\vone}{\boldsymbol{1}}

\newcommand{\cG}{\mathcal{G}}
\newcommand{\cH}{\mathcal{H}}

\newcommand{\mA}{\mathsf{A}}
\newcommand{\mtA}{\tilde{\mathsf{A}}}
\newcommand{\mPsi}{\mathsf{\Psi}}
\newcommand{\mPhi}{\mathsf{\Phi}}
\newcommand{\mOmega}{\mathsf{\Omega}}

\newcommand{\tpsi}{\tilde{\psi}}

\def\oneshot#1{\mathop{\mathrm{#1}}\limits}

\title{A Two-Dimensional Inverse Frame Operator Approximation Technique \thanks{This work is supported in part by grants NSF-DMS 1216559, NSF-DMS 1521600, NSF-DMS 1521661, NSF 1502640,  and AFOSR FA9550-15-1-0152.}}

\author{Guohui Song \footnotemark[1]
\and Jacqueline Davis \footnotemark[2] 
\and Anne Gelb \footnotemark[2]
}

\begin{document}
\date{}
\maketitle

\renewcommand{\thefootnote}{\fnsymbol{footnote}}

\footnotetext[1]{Department of Mathematics, Clarkson University, Potsdam, NY 13699. (\email{gsong@clarkson.edu})}
\footnotetext[2]{School of Mathematical and Statistical Sciences,
         Arizona State University,
         P.O. Box 871804,
         Tempe, AZ 85287-1804. (\email{Jacqueline.T.Davis@asu.edu}) (\email{anne.gelb@asu.edu})}

\begin{abstract}
The ability to efficiently and accurately construct an inverse frame operator is critical for establishing the utility of numerical frame approximations.  Recently, the {\em admissible frame} method was developed to approximate inverse frame operators for one-dimensional problems. Using the admissible frame approach, it is possible to project the corresponding frame data onto a more suitable (admissible) frame, even when the sampling frame is only weakly localized.  As a result, a target function may be approximated as a finite frame expansion with its asymptotic convergence solely dependent on its smoothness.  In this investigation, we seek to expand the admissible frame approach  to two dimensions, which requires some additional constraints.  We prove that the admissible frame technique converges in two dimensions and then demonstrate its usefulness with some numerical experiments that use sampling patterns inspired by applications that sample data non-uniformly in the Fourier domain.
\end{abstract}



\section{Introduction}
\label{sec:introduction}

In applications such as magnetic resonance imaging (MRI) and synthetic aperture radar (SAR), data may be sampled as non-uniform Fourier coefficients (see e.g. \cite{BW00,FPC,GL,pipe,sedarat,Aditya09}).  The current methodology for recovering an underlying image or extracting features from the image often involves interpolating or approximating the given data to obtain integer Fourier coefficients and then applying the inverse FFT.  The combined process is often referred to as the non-uniform FFT (NFFT), \cite{DR,FPC,GL,Fourmont}, or the convolutional gridding algorithm, \cite{pipe}.  Although the NFFT is computationally efficient, there can be significant issues regarding its robustness.  Indeed, it was shown in \cite{GelbSong2014, Aditya09} that there are sampling patterns for which the NFFT fails to converge, regardless of the smoothness properties of the underlying function.  Moreover, since the target image is typically only piecewise smooth, the Gibbs phenomenon causes additional error that must also be addressed.  In \cite{GelbHines,GelbSong2014,Song2011b}, the non-uniform Fourier reconstruction problem was recast as a finite {\em Fourier frame approximation}, \cite{DS52}, in efforts to combat the ill-conditioning caused by the resampling of Fourier data.  We remark that there are an abundance of algorithms that recover piecewise smooth images from non-uniform Fourier data. For example, techniques that use $l^1$ regularization are becoming increasingly prevalent, \cite{LDP}.  While such algorithms have their advantages, their reliance on iterative solvers make them somewhat difficult to analyze.  Further, we anticipate that the admissible frame approach described here may be used to initialize iterative algorithms, and ultimately make them more efficient and robust.

One of the main difficulties in approximating a function from its frame coefficients, independent of the smoothness properties of the underlying image, lies in the construction of an accurate and efficient approximation to the inverse frame operator.  The frame algorithm designed in \cite{DS52} and accelerated in  \cite{MR1955936, Grochenig1993} is iterative, with its speed inherently depending on the frame bounds.


An alternative algorithm for constructing the inverse frame operator was developed for one-dimensional problems in \cite{Song2011b},  notably for Fourier frames.  It was established there that sampling with {\em localized frames} improves both the accuracy of the numerical frame approximation as well as the robustness and efficiency of the (finite) frame operator inversion. Moreover, in applications where the given data may not constitute a localized frame, a technique was devised to project the corresponding frame data onto a more suitable frame.  As a result, the target function may be approximated as a finite expansion with its asymptotic convergence solely dependent on its smoothness.  We note that if the target function is only piecewise smooth, the resulting Fourier frame approximation could be post-processed to achieve a high order reconstruction in smooth regions, \cite{GelbHines}.  In the current investigation we seek to expand these results to two dimensions, which requires some additional constraints.  Following the approach in \cite{Song2011b}, we first analyze the convergence properties of the Casazza-Christensen method for approximating the two-dimensional inverse frame operator for localized frames, \cite{Casazza2000, Christensen2000}. We then consider the case when the frame is only {\em weakly localized}, that is, the localization rate $s$ in Definition \ref{def:2dlocalized} is less than 2. In this case, the Casazza-Christensen algorithm may be very slow to converge.  The numerical algorithm developed in \cite{Song2011b} uses the {\em admissible frames} approach to approximate  $S^{-1}$ for one-dimensional weakly localized frames. In addition, this technique yields faster convergence than the Casazza-Christensen method. Thus in this paper we are also motivated to extend the the admissible frame methodology to two dimensions for the weakly localized case.

The generalization of the admissible frame methodology to two dimensions will be useful in a variety of applications.  For example, some data acquisition techniques in MRI favor collecting Fourier samples non-uniformly, such as in a spiral \cite{pipe, sedarat} or rosette trajectory, \cite{Noll97, noll1998}. Synthetic aperture radar (SAR) acquisition data can also be described as non-uniform Fourier samples, \cite{FPC}.  Moreover, even in cases where the collection scheme is presumably uniform, machine error causes the samples to be ``jittered'' away from this distribution.   Such situations infer that the data may be viewed as coefficients of weakly localized Fourier frames.

The paper is organized as follows:
Section \ref{sec:review} reviews some fundamental aspects of frame theory in two dimensions.  In Section \ref{sec:localizedframes} we establish the convergence rate of the Casazza-Christensen method of approximating the inverse frame operator for localized frames, \cite{Casazza2000, Christensen2000}. However, the convergence rate fails to hold when the sampling frame is only weakly localized.  To overcome this difficulty we propose a new method of approximating the inverse frame operator and prove its convergence rate in Section \ref{sec:admissibleframes}.  In Section \ref{sec:numerical} we use this approximation technique to develop a new numerical frame approximation method.  We demonstrate the effectiveness of our method  with some numerical experiments.  Concluding remarks are provided in Section \ref{sec:conclusion}.

\section{Two-Dimensional Frames}
\label{sec:review}
Below we provide a brief introduction to the problem of sampling with frames in two dimensions, starting with the definition of a two-dimensional frame.

\begin{defn}
\label{def:frame}
Suppose $\cH$ is a separable Hilbert space. We say $\{\psi_{\vj}: \vj=(j_1, j_2)\in \bN^2\} \subseteq \cH$ is a frame for $\cH$ if for there exist some positive constants $A$ and $B$ such that all $f\in \cH$
\begin{equation*}
A\|f\|^2 \leq \sum_{\vj\in \bN^2}\bigl|\langle f, \psi_{\vj} \rangle\bigr|^2 \leq B\|f\|^2.
\end{equation*}
The associated frame operator $S:\cH\rightarrow \cH$ is defined as
\begin{equation*}
Sf = \sum_{\vj\in \bN^2} \langle f, \psi_{\vj} \rangle \psi_{\vj}, \quad f\in \cH,
\end{equation*}
where $\langle f, \psi_{\vj} \rangle$ are often referred to as frame coefficients. 
\end{defn}

To recover the underlying function $f$ from its frame coefficients, we have
\begin{equation}
f = S^{-1}Sf = \sum_{\vj\in \bN^2} \langle f, \psi_{\vj} \rangle \tpsi_{\vj},
\label{eq:f_frame}
\end{equation}
where $\tpsi_{\vj}=S^{-1}\psi_{\vj}$ is called the (canonical) dual frame.

In applications such as magnetic resonance imaging (MRI) or synthetic aperture radar (SAR), we typically seek to recover an image or features of an image from its (finite) frame coefficients, which requires an efficient algorithm for approximating the dual frame or the inverse frame operator $S^{-1}$, usually not available in closed form.\footnote{For applications such as MRI or SAR, the data are collected as non-uniform Fourier samples, and a variety of techniques have been developed for image reconstruction in these modalities. As described in \cite{GelbSong2014}, forward algorithms such as the non-uniform FFT can be considered as approximations to the finite Fourier frame approximation.}  Many numerical methods exist for approximating  (\ref{eq:f_frame}).  For example, a general method for approximating $S^{-1}$, which we will refer to as the Casazza-Christensen method,  was proposed in \cite{Casazza2000}, and its convergence was proved in \cite{Casazza2000, Christensen2000, Christensen2003} for any general frame.  However, obtaining a {\em rate} of convergence requires some extra conditions on the frames.  One such condition, used in \cite{Song2011b} to acquire the convergence rate of the Casazza-Christensen method for one-dimensional problems,  is the {\it localization} of frames, and is defined as

 \begin{defn}
\label{def:localized}
 Suppose $\{\psi_j: j\in \bN\}$ is a frame in $\cH$. We say that $\{\psi_j: j\in \bN\}$ is {\em intrinsically (self-) localized} if there exists a positive constant $c$ such that
\begin{equation*}
\left|\langle \psi_j, \psi_l \rangle\right| \leq c (1+ |j-l|)^{-s}, \quad  j,l\in \bN,
\end{equation*}
with $s>1$.
\end{defn}

The analogous definition in two dimensions, which will be used to generalize the convergence analysis for the Casazza-Christensen method is

\begin{defn}\label{def:2dlocalized}
 Suppose $\{\psi_{\vj}: \vj \in \bN^2\}$ is a frame in $\cH$. We say that $\{\psi_{\vj}: \vj \in \bN^2\}$ is {\em intrinsically (self-) localized} if there exist positive constants $\gamma>0$ and $s>2$ such that
\begin{equation}\label{eq:localizedcondition}
\left|\langle \psi_{\vj}, \psi_{\vl} \rangle\right| \leq \gamma (1+ \|\vj-\vl\|_2)^{-s}, \quad  \vj, \vl\in \bN^2,
\end{equation}
where $\|\cdot\|_2$ denotes the Euclidean norm.
\end{defn}


\section{Approximating $S^{-1}$ with localized frames in 2D}\label{sec:localizedframes}
We will assume in this section that the frame $\{\psi_{\vj}: \vj \in \bN^2\}$ is intrinsically localized, that is, it satisfies \eqref{eq:localizedcondition} for some $s>2$. We will analyze the convergence properties, in particular, provide a convergence rate, for the Casazza-Christensen approximation of the inverse frame operator $S^{-1}$. To this end, we first briefly review this method in two-dimensional notation.

For any $\vn=(n_1, n_2)\in \bN^2$, let $\cH_{\vn}:=\spn\{\psi_{\vj}: \vone\leq \vj\leq \vn\}$, where $\vone\leq \vj\leq \vn$ denotes $1\leq j_1\leq n_1$ and $1\leq j_1\leq n_2$, be the finite-dimensional subspace of $\cH$. Observe that $\{\psi_{\vj}: \vone\leq \vj\leq \vn\}$ is a frame for $\cH_\vn$ (c.f. \cite{Christensen2003}) with the frame operator $S_{\vn}: \cH \rightarrow \cH_{\vn}$ given by
\begin{equation}\label{finiteframe}
S_{\vn} f:= \sum_{\vone\leq \vj\leq \vn} \langle f, \psi_{\vj} \rangle \psi_{\vj}, \quad f\in \cH.
\end{equation}
For $\vm, \vn\in \bN^2$, suppose $P_\vn$ is the orthogonal projection from $\cH$ onto $\cH_\vn$ and $V_{\vn,\vm}$ is the restriction of $P_\vn S_\vm$ to $\cH_\vn$:
\begin{equation}
\label{eq:V_nm}
V_{\vn, \vm}:= P_\vn S_\vm |_{\cH_\vn}.
\end{equation}
Note that for any $\vn\in \bN$, we can always find large enough $\vm(\vn)\in \bN^2$ (with both $m_1$ and $m_2$ large enough) depending on $\vn$ such that $V_{\vn, \vm(\vn)}$ is invertible on $\cH_\vn$, \cite{Casazza2000}. The Casazza-Christensen method is to use $V_{\vn, \vm(\vn)}^{-1}P_\vn$ as an approximation of the inverse frame operator $S^{-1}$. It was shown in \cite{Casazza2000} that there exists $\vm(\vn)\in \bN^2$ depending $\vn$ such that for any $f\in \cH$,
\begin{equation}
\label{eq:error}
\|S^{-1}f - V_{\vn, \vm(\vn)}^{-1}P_\vn f \| \rightarrow 0, \quad \mbox{ as } \vn\rightarrow \infty,
\end{equation}
which requires both $n_1 \rightarrow \infty$ and $n_2 \rightarrow \infty$.  Although (\ref{eq:error}) was only proven in one dimension, it is trivial to extend the result to higher dimensions.  Moreover, the Casazza-Christensen method works for any general frame without any extra conditions.  However, there are still two remaining issues  to be addressed. First, we only know for any $\vn\in \bN^2$, there exists a large enough $\vm(\vn)$ such that the convergence holds, but there is no explicit formula for computing $\vm(\vn)$, which is needed in practical applications. Second, no convergence rate is given for (\ref{eq:error}).

Below we propose an explicit formula to compute $\vm(\vn)$ and also estimate the convergence rate of the Casazza-Christensen method.  As in the one-dimensional case discussed in \cite{Song2011b}, we will employ the localization properties of localized frames. To this end, we impose the following smoothness assumption on $f$:
\begin{defn}\label{def:smooth_assum}
The function $f$ satisfies the smoothness assumption with respect to a frame $\{\psi_{\vj}: \vj \in \bN^2\}$  if for $s > 2$ there exists a positive constant $c_0$ such that
\begin{equation}\label{eq:smoothassumption}
\bigl| \langle f, \psi_{\vj} \rangle \bigr| \leq c_0 \|\vj\|_{2}^{-s}, \quad \vj\in \bN^{2}.
\end{equation}
\end{defn}
Observe from the the localization property \eqref{eq:localizedcondition} that for any $\vl\in \bN^2$, $\psi_\vl$ satisfies (\ref{eq:smoothassumption}). Therefore the smoothness assumption is also reasonable for $f$ since the main objective is to approximate (\ref{eq:f_frame}).

We begin by decomposing (\ref{eq:error}) as
\begin{eqnarray}
\| S^{-1}f - V_{\vn, \vm}^{-1}P_{\vn} f\|
&\leq &  \|S^{-1}f  - P_{\vn} S^{-1} f\|
+ \|P_{\vn} S^{-1} f - V_{\vn, \vm}^{-1}P_{\vn}S_{\vm} S^{-1} f \|\nonumber\\
&+& \| V_{\vn, \vm}^{-1}P_{\vn}S_{\vm} S^{-1} f - V_{\vn, \vm}^{-1}P_{\vn} f\|.
\label{eq:error_decomp1}
\end{eqnarray}
Since  $V_{\vn, \vm}^{-1} P_{\vn}S_{\vm} g=g$ for $g\in \cH_n$, the second term on the right hand side of (\ref{eq:error_decomp1}) can be reformulated as
\begin{equation*}
\|P_{\vn} S^{-1} f - V_{\vn, \vm}^{-1}P_{\vn}S_{\vm} S^{-1} f \| = \|V_{\vn,\vm}^{-1}P_{\vn}S_{\vm}(P_{\vn} S^{-1}f -S^{-1}f)  \|,
\end{equation*}
while the third term can be written as
\begin{equation*}
\| V_{\vn, \vm}^{-1}P_{\vn}S_{\vm} S^{-1} f - V_{\vn, \vm}^{-1}P_{\vn} f\| = \|V_{\vn,\vm}^{-1}P_{\vn} (S_{\vm} - S)S^{-1}f \|.
\end{equation*}
It therefore follows that
\begin{eqnarray}\label{error-decomposition-inv}
\| S^{-1}f - V_{\vn,\vm}^{-1}P_{\vn} f\|
&\leq&  \|S^{-1}f  - P_{\vn} S^{-1} f\|\nonumber\\
&+& \|V_{\vn,\vm}^{-1}P_{\vn}S_{\vm}(S^{-1}f -P_{\vn} S^{-1}f)  \|+  \|V_{\vn,\vm}^{-1}P_{\vn} (S- S_{\vm})S^{-1}f \|.
\end{eqnarray}

The first term in \eqref{error-decomposition-inv} can be viewed as the projection error of $S^{-1}f$.   We need the following two lemmas to estimate its error.

\begin{lem}\label{lem:dualcoefficients}
If $\{\psi_\vj\}$ is a localized frame as in \eqref{eq:localizedcondition} and assumption \eqref{eq:smoothassumption} holds, then there exists a positive constant $c$ such that
\begin{equation*}
\bigl| \langle f, \tpsi_\vj \rangle \bigr| \leq c \|\vj\|_{2}^{-s}.
\end{equation*}
\end{lem}
\begin{proof}
We first derive an estimate on $\langle \tpsi_\vj, \tpsi_\vl \rangle$ through Wiener's Lemma. Let $\mA$ be the infinite Grammian $[\langle \psi_\vj, \psi_\vl \rangle]_{\vj, \vl\in \bN^2}$ and $\mtA$ be the infinite Grammian $[\langle \tpsi_\vj, \tpsi_\vl \rangle]_{\vj, \vl\in \bN^2}$. It was shown in \cite{Fornasier2005} that $\mtA=(\mA^\dagger)^2\mA$, where $\mA^\dagger$ is the pseudo-inverse of $\mA$. By Wiener's Lemma on matrices with polynomial decay \cite{Fornasier2005, Sun2005, Sun2007, Sun2011}, the infinite Grammian matrix $\mtA$ has the same polynomial decay as $\mA$. Thus there exists a positive constant $c_1$ such that
\begin{equation}\label{eq:dualframedecay}
\left|\langle \tpsi_{\vj}, \tpsi_{\vl} \rangle\right| \leq c_1 (1+ \|\vj-\vl\|_2)^{-s}, \quad  \vj, \vl\in \bN^2.
\end{equation}

An estimate of $\langle f, \tpsi_\vj \rangle$ follows from the dual frame property that for $f= \sum_{\vl\in \bN^2} \langle f, \psi_\vl \rangle \tpsi_\vl$, we have
\begin{equation*}
\langle f, \tpsi_\vj \rangle = \sum_{\vl\in \bN^2} \langle f, \psi_\vl \rangle \langle \tpsi_\vl, \tpsi_\vj \rangle,
\end{equation*}
for any $\vj\in \bN^2.$  Substituting \eqref{eq:smoothassumption} and \eqref{eq:dualframedecay} into the above equality yields
\begin{equation}\label{eq:dualcoeffest}
\bigl| \langle f, \tpsi_\vj \rangle \bigr| \leq c_0 c_1 \sum_{\vl\in \bN^2} \|\vl\|_2^{-s} (1+ \|\vj-\vl\|_2)^{-s}.
\end{equation}
We next estimate the summation term in the above inequality using a standard technique employed in the proof of Wiener's Lemma \cite{Grochenig2004, Sun2005}. By triangular inequality, for any $\vj\in \bN^2$, $1+ \|\vj-\vl\|_2 + \|\vl\|_2 \geq \|\vj\|_2$. Thus for any $\vj\in \bN^2$, either $1+ \|\vj-\vl\|_2$ or $\|\vl\|_2$ is no less than $\frac{1}{2}\|\vj\|_2$. Consequently,
\begin{equation*}
\sum_{\vl\in \bN^2} \|\vl\|_2^{-s} (1+ \|\vj-\vl\|_2)^{-s} \leq \sum_{\vl\in \bN^2} 2^s \|\vj\|_2^{-s} (1+ \|\vj-\vl\|_2)^{-s} + \sum_{\vl\in \bN^2} \|\vl\|_2^{-s} 2^s \|\vj\|_2^{-s}.
\end{equation*}
Moreover, when $s>2$, both $\sum_{\vl\in \bN^2}(1+ \|\vj-\vl\|_2)^{-s}$ and $\sum_{\vl\in \bN^2} \|\vl\|_2^{-s}$ are summable and in particular are bounded by a fixed constant. Thus there exists a positive constant $c_2$ such that
$
\sum_{\vl\in \bN^2} \|\vl\|_2^{-s} (1+ \|\vj-\vl\|_2)^{-s} \leq c_2 \|\vj\|_2^{-s}.
$
The desired result follows from substituting it into \eqref{eq:dualcoeffest}.
\end{proof}

Lemma \ref{lem:dualcoefficients} yields an upper bound for the frame coefficients with respect to the dual frame, while Lemma \ref{lem:projectionerror} provides an estimate on the projection error for any $f$ satisfying the smooth assumption \eqref{eq:smoothassumption}.

\begin{lem}\label{lem:projectionerror}
If $\{\psi_\vj\}$ is a localized frame as in \eqref{eq:localizedcondition} and assumption \eqref{eq:smoothassumption} holds, then there exists a positive constant $c$ such that for any $\vn\in \bN^2$
\begin{equation*}
\|f- P_\vn f\|\leq c n_1^{-\frac{s-1}{2}} + c n_2^{-\frac{s-1}{2}}.
\end{equation*}
\end{lem}
\begin{proof}
For $\vn\in \bN^2$, let $T_\vn f =\sum_{1\leq \vj\leq \vn} \langle f, \tpsi_\vj \rangle \psi_\vj$. Since $T_\vn f\in \cH_\vn$ and  $P_\vn$ is the orthogonal projection onto $\cH_\vn$, we have
$
\|f- P_\vn f\| \leq \|f- T_\vn f\|.
$ It is enough to show $\|f- T_\vn f\|$ satisfies the desired inequality. 

To estimate $\|f- T_\vn f\|$, we use the dual frame property $f=\sum_{\vj\in \bN^2} \langle f, \tpsi_\vj \rangle \psi_\vj$ which implies
\begin{equation}\label{eq:projerrordecomp}
\|f- T_\vn f \| \leq \biggl\| \sum_{j_1>n_1, j_2\in \bN} \langle f, \tpsi_\vj \rangle \psi_\vj \biggr\| + \biggl\| \sum_{j_1\in \bN, j_2>n_2} \langle f, \tpsi_\vj \rangle \psi_\vj \biggr\|.
\end{equation}
Considering the first  term on the right hand, we have
\begin{equation*}
\biggl\| \sum_{j_1>n_1, j_2\in \bN} \langle f, \tpsi_\vj \rangle \psi_\vj \biggr\|^2 = \sum_{j_1>n_1, j_2\in \bN} \sum_{l_1>n_1, l_2\in \bN} \langle f, \tpsi_\vj \rangle \langle f, \tpsi_\vl \rangle \langle \psi_\vj, \psi_\vl \rangle.
\end{equation*}
Substituting the localization condition \eqref{eq:localizedcondition} and the smooth assumption \eqref{eq:smoothassumption} into the above equality, yields for any $\vn\in \bN^2$,
\begin{equation*}
\biggl\| \sum_{j_1>n_1, j_2\in \bN} \langle f, \tpsi_\vj \rangle \psi_\vj \biggr\|^2 \leq c_1 \sum_{j_1>n_1, j_2\in \bN} \sum_{l_1>n_1, l_2\in \bN} \|\vj\|_2^{-s} \|\vl\|_2^{-s} (1+ \|\vj-\vl\|_2)^{-s},
\end{equation*}
where $c_1$ is a positive constant. Note that for any $\va,\vb\in \bN^2$, we have $\|\va\|_2 \geq \sqrt{2 a_1 a_2}$ and $1 + ||\va -\vb||_2 \ge \sqrt{1 + |a_1-b_1|}\sqrt{1+|a_2-b_2|}$. Thus
\begin{equation*}
\biggl\| \sum_{j_1>n_1, j_2\in \bN} \langle f, \tpsi_\vj \rangle \psi_\vj \biggr\|^2 \leq c_1 2^{-s} \sum_{j_1>n_1, l_1 >n_1} j_1^{-\frac{s}{2}} l_1^{-\frac{s}{2}} (1+ |j_1 - l_1|)^{-\frac{s}{2}} \sum_{j_2\in \bN, l_2\in \bN} j_2^{-\frac{s}{2}} l_2^{-\frac{s}{2}} (1+ |j_2 - l_2|)^{-\frac{s}{2}}.
\end{equation*}
Since $\sum_{j_2\in \bN, l_2\in \bN} j_2^{-\frac{s}{2}} l_2^{-\frac{s}{2}} (1+ |j_2 - l_2|)^{-\frac{s}{2}}$ is bounded whenever $s>2$, which is required for the localization property in Definition \ref{def:2dlocalized}, it follows that there exists a positive constant $c_2$ such that
\begin{equation*}
\biggl\| \sum_{j_1>n_1, j_2\in \bN} \langle f, \tpsi_\vj \rangle \psi_\vj \biggr\|^2 \leq c_2 \sum_{j_1>n_1, l_1 >n_1} j_1^{-\frac{s}{2}} l_1^{-\frac{s}{2}} (1+ |j_1 - l_1|)^{-\frac{s}{2}}.
\end{equation*}
It was shown in \cite{Song2011b} that the summation term on the right hand side of the above inequality is bounded by $\frac{2}{(s/2-1)(s-1)} n_1^{-(s-1)}$, yielding
\begin{equation*}
\biggl\| \sum_{j_1>n_1, j_2\in \bN} \langle f, \tpsi_\vj \rangle \psi_\vj \biggr\|^2 \leq c_3 n_1^{-(s-1)},
\end{equation*}
where $c_3=\frac{2 c_2}{(s/2-1)(s-1)}$. For the second term in \eqref{eq:projerrordecomp} we could similarly show that there exists a positive constant $c_4$ such that
\begin{equation*}
\biggl\| \sum_{j_1\in \bN, j_2>n_2} \langle f, \tpsi_\vj \rangle \psi_\vj \biggr\|^2 \leq c_4 n_2^{-(s-1)}.
\end{equation*}
The desired result then follows from substituting the above two inequalities into \eqref{eq:projerrordecomp}.
\end{proof}

We are now ready to present an estimate of the first term in the error decomposition \eqref{error-decomposition-inv}.

\begin{prop}\label{prop:projectionerror}
If $\{\psi_\vj\}$ is a localized frame as in \eqref{eq:localizedcondition} and the smoothness assumption \eqref{eq:smoothassumption} holds, then there exists a positive constant $c$ such that for any $\vn\in \bN^2$
\begin{equation*}
\|S^{-1}f- P_\vn S^{-1}f\| \leq c n_1^{-\frac{s-1}{2}} + c n_2^{-\frac{s-1}{2}}.
\end{equation*}
\end{prop}
\begin{proof}
Since the frame operator $S$ is self-adjoint, we have $\langle S^{-1}f, \psi_\vj \rangle = \langle f, S^{-1}\psi_\vj \rangle = \langle f, \tpsi_\vj \rangle$. By Lemma \ref{lem:dualcoefficients}, we know that $S^{-1}f$ also satisfies the smoothness assumption \eqref{eq:smoothassumption}. The desired result therefore follows from applying Lemma \ref{lem:projectionerror} on $S^{-1}f$.
\end{proof}

We next estimate $\|V_{\vn,\vm}^{-1}P_{\vn}S_{\vm}(S^{-1}f -P_{\vn} S^{-1}f) \|$, the second term in the error decomposition \eqref{error-decomposition-inv}. It it evident that $\|P_\vn\| \leq 1$ and $\|S_\vm\| \leq \|S\| \leq B$, the upper frame bound for $\{\psi_\vj\}$, for $\vm, \vn\in \bN^2$. It remains to estimate $\|V_{\vn,\vm}^{-1}\|$. Specifically, we will show an explicit choice of $\vm$ depending on $\vn$ such that $\|V_{\vn,\vm}^{-1}\|$ is uniformly bounded for all $\vn\in \bN^2$. To this end, we define the following constant
\begin{equation*}
A_{\vm, \vn} = \frac{2s \gamma^2}{ (s-1)^2 \lambda_{\min}(\mPsi_\vn)} n_1 n_2 \bigl[ (m_1 -n_1)^{-(s-1)} +  (m_2 -n_2)^{-(s-1)} \bigr],
\end{equation*}
where $\mPsi_\vn =\bigl[ \langle \psi_\vj, \psi_\vl \bigr]_{\vone\leq \vj,\vl\leq \vn}$ and $\lambda_{\min}(\mPsi_\vn)$ is its smallest eigenvalue. We assume here that the matrix $\mPsi_\vn$ is invertible.  Otherwise, we  can use its invertible principle submatrix instead. We always choose $m_1, m_2$ larger than $n_1, n_2$ respectively.

Lemma \ref{lem:mchoice} provides an upper bound of $\|V_{\vn,\vm}^{-1}\|$.
\begin{lem}\label{lem:mchoice}
Suppose $\{\psi_\vj\}$ is a localized frame as in \eqref{eq:localizedcondition}. If $A_{\vm, \vn} <A$, the lower frame bound constant for $\{\psi_\vj\}$, then $\{P_\vn \psi_\vj: \vone\leq \vj\leq \vm\}$ is a frame for $\cH_\vn$ with frame bound constants $A-A_{\vm,\vn}$, $B$, and the corresponding frame operator is $V_{\vn, \vm}$ defined in \eqref{eq:V_nm}. Moreover, we have
\begin{equation*}
\|V_{\vn,\vm}^{-1}\| \leq \frac{1}{A-A_{\vm, \vn}}.
\end{equation*}
\end{lem}
\begin{proof}
We proceed by establishing the frame condition by direct computation. Suppose $g\in \cH_\vn$. We need to find a lower and upper bound for  $\sum_{\vone\leq \vj\leq \vm} |\langle g, P_\vn \psi_\vj \rangle|^2$. Note that the projection operator $P_\vn$ is self-adjoint, implying $\langle g, P_\vn \psi_\vj \rangle = \langle P_\vn g,  \psi_\vj \rangle$. Moreover, since $g\in \cH_\vn$, we have $P_\vn g =g$. It follows that
\begin{equation*}
\sum_{\vone\leq \vj\leq \vm} |\langle g, P_\vn \psi_\vj \rangle|^2 = \sum_{\vone\leq \vj\leq \vm} |\langle g, \psi_\vj \rangle|^2
\end{equation*}
An upper bound follows from the simple observation that for any $\vm\in \bN^2$,
\begin{equation*}
\sum_{\vone\leq \vj\leq \vm} |\langle g, \psi_\vj \rangle|^2 \leq  \sum_{j_1=1}^{\infty}\sum_{j_2=1}^{\infty} |\langle g, \psi_\vj \rangle|^2 \leq B\|g\|^2.
\end{equation*}
We next determine a lower bound for $\sum_{\vone\leq \vj\leq \vm} |\langle g, \psi_\vj \rangle|^2$. To this end, we first write
\begin{equation}
\label{eq:finitesum}
\sum_{\vone\leq \vj\leq \vm} |\langle g, \psi_\vj \rangle|^2 = \sum_{\vj \in \bN^2}|\langle g, \psi_\vj \rangle|^2 - \widetilde{\sum} |\langle g, \psi_\vj \rangle|^2,
\end{equation}
where we have adopted the notation
$\widetilde{\sum} = \sum_{\{\vj\in \bN^2\backslash [1, m_1]\times [1, m_2]\}}.$
We proceed by determining an upper bound for $\widetilde{\sum}|\langle g, \psi_\vj \rangle|^2$.  Observe that
\begin{equation}\label{eq:tail}
\widetilde{\sum}|\langle g, \psi_\vj \rangle|^2 \leq \sum_{j_1>m_1}\sum_{j_2\in \bN} |\langle g, \psi_\vj \rangle|^2 + \sum_{j_1\in \bN}\sum_{j_2>m_2} |\langle g, \psi_\vj \rangle|^2.
\end{equation}
Examining the first term on the right hand side of (\ref{eq:tail}), we note that since $g\in \cH_\vn$, it can be written as  $g=\sum_{\vone\leq \vl\leq \vn} a_\vl \psi_\vl$ for some $\va=(a_\vl: \vone\leq \vl\leq \vn)$ viewed as a $n_1 n_2-$vector. It follows that
$\oneshot{\sum}_{j_1 > m_1} \oneshot{\sum}_{j_2\in \bN} |\langle g, \psi_\vj \rangle|^2 = \oneshot{\sum}_{j_1>m_1}\oneshot{\sum}_{j_2\in \bN} \biggl| \oneshot{\sum}_{\vone\leq \vl\leq \vn}a_\vl \langle \psi_\vl, \psi_\vj \rangle \biggr|^{2}.$
By Cauchy-Schwartz inequality, we have
\begin{equation*}
\sum_{j_1>m_1}\sum_{j_2\in \bN} |\langle g, \psi_\vj \rangle|^2 \leq \|\va\|_2^2 \sum_{j_1>m_1}\sum_{j_2\in \bN} \sum_{\vone\leq \vl\leq \vn} \bigl|\langle \psi_\vl, \psi_\vj \rangle \bigr|^2.
\end{equation*}
Applying the localization condition \eqref{eq:localizedcondition} to the above inequality yields
\begin{equation*}
\sum_{j_1>m_1}\sum_{j_2\in \bN} |\langle g, \psi_\vj \rangle|^2 \leq \gamma^2 \|\va\|_2^2 \sum_{j_1>m_1}\sum_{j_2\in \bN} \sum_{\vone\leq \vl\leq \vn} (1+\|\vj-\vl\|_2)^{-2s}.
\end{equation*}
As in Lemma \ref{lem:projectionerror}, we use the fact that $1+\|\vj-\vl\|_2 \geq (1+|j_1-l_1|)^{1/2} (1+|j_2-l_2|)^{1/2}$ to obtain
\begin{equation}\label{eq:taildecomp}
\sum_{j_1>m_1}\sum_{j_2\in \bN} |\langle g, \psi_\vj \rangle|^2 \leq \gamma^2 \|\va\|_2^2 \sum_{j_1>m_1}\sum_{l_1=1}^{n_1} (1+|j_1-l_1|)^{-s} \sum_{j_2\in \bN}\sum_{l_2=1}^{n_2} (1+|j_2-l_2|)^{-s}.
\end{equation}
When $s>2$, for any $l_2\in \bN$, we have $\sum_{j_2\in \bN} (1+|j_2-l_2|)^{-s} \leq \frac{2s}{s-1}$. That is,
\begin{equation*}
\sum_{j_2\in \bN}\sum_{l_2=1}^{n_2} (1+|j_2-l_2|)^{-s} \leq \frac{2s n_2}{s-1}.
\end{equation*}
Moreover, it follows from the proof of Lemma 3.4 in \cite{Song2011b} that
\begin{equation*}
\sum_{j_1>m_1}\sum_{l_1=1}^{n_1} (1+|j_1-l_1|)^{-s} \leq \frac{n_1}{s-1} (m_1-n_1)^{-(s-1)}.
\end{equation*}
Substituting the above two inequalities into \eqref{eq:taildecomp} yields
\begin{equation*}
\sum_{j_1>m_1}\sum_{j_2\in \bN} |\langle g, \psi_\vj \rangle|^2 \leq \|\va\|_2^2 \frac{2s \gamma^2}{ (s-1)^2} n_1 n_2 (m_1 -n_1)^{-(s-1)}.
\end{equation*}
For the second term on the right hand side of (\ref{eq:tail}) we similarly have
\begin{equation*}
\sum_{j_1\in \bN}\sum_{j_2>m_2} |\langle g, \psi_\vj \rangle|^2 \leq \|\va\|_2^2 \frac{2s \gamma^2}{ (s-1)^2} n_1 n_2 (m_2 -n_2)^{-(s-1)}.
\end{equation*}
Combining the above two inequalities provides the upper bound
\begin{equation*}
\widetilde{\sum}|\langle g, \psi_\vj \rangle|^2 \leq \|\va\|_2^2 \frac{2s \gamma^2}{ (s-1)^2} n_1 n_2 \bigl[ (m_1 -n_1)^{-(s-1)} +  (m_2 -n_2)^{-(s-1)} \bigr].
\end{equation*}
For the first term on the right hand side of (\ref{eq:finitesum}), we have $\|g\|^2 = \|\sum_{\vone\leq \vl\leq \vn} a_\vl \psi_\vl\|^2 = \va^T \mPsi_\vn \va \geq \lambda_{\min}(\mPsi_\vn) \|\va\|_2^2$, that is
$\|\va\|^2_2 \leq \frac{\|g\|^2}{\lambda_{\min}(\mPsi_\vn)}$. Thus
\begin{equation*}
\widetilde{\sum}|\langle g, \psi_\vj \rangle|^2 \leq A_{\vm, \vn} \|g\|^2.
\end{equation*}
Finally, since $\sum_{\vj\in \bN^2}|\langle g, \psi_\vj \rangle|^2 \geq A \|g\|^2$ by the frame condition, we obtain a lower bound using (\ref{eq:finitesum}) for any $g\in \cH_\vn$ as
\begin{equation*}
\sum_{\vone\leq \vj\leq \vm} |\langle g, \psi_\vj \rangle|^2 \geq (A - A_{\vm,\vn}) \|g\|^2,
\end{equation*}
or in general, $A-A_{\vm, \vn}$ is a lower frame bound constant for the frame $\{P_\vn \psi_\vj: \vone\leq \vj\leq \vm\}$ if $A_{\vm, \vn} <A$.
Moreover, we observe that for $g\in \cH_\vn$,
\begin{equation*}
\sum_{\vone\leq \vj\leq \vm} \langle g, P_\vn \psi_\vj\rangle P_\vn \psi_\vj  = \sum_{\vone\leq \vj\leq \vm} \langle g,  \psi_\vj\rangle P_\vn \psi_\vj = P_\vn S_\vm g = V_{\vn, \vm}g.
\end{equation*}
That is, $V_{\vn, \vm}$ is the corresponding frame operator and the bound of $\|V_{\vn,\vm}^{-1}\|$ follows immediately.
\end{proof}

We next present an estimate for $\|V_{\vn,\vm}^{-1}P_{\vn}S_{\vm}(S^{-1}f -P_{\vn} S^{-1}f) \|$ the second term in the error decomposition \eqref{error-decomposition-inv} by choosing $\vm$ such that $\|V_{\vn,\vm}^{-1}\|$ is uniformly bounded for all $\vn$.

\begin{prop}\label{prop:secondterm}
Suppose $\{\psi_\vj\}$ is a localized frame as in \eqref{eq:localizedcondition} and assumption \eqref{eq:smoothassumption} holds. If
\begin{equation}\label{eq:mchoice}
m_1 = n_1 +  \frac{8s \gamma^2 n_1 n_2}{A (s-1)^2 \lambda_{\min}(\mPsi_\vn)} , \quad m_2 = n_2 +  \frac{8s \gamma^2 n_1 n_2}{A (s-1)^2 \lambda_{\min}(\mPsi_\vn)},
\end{equation}
then $\|V_{\vn, \vm}^{-1}\| \leq \frac{2}{A}$ and there exist s a positive constant such that
\begin{equation*}
\|V_{\vn,\vm}^{-1}P_{\vn}S_{\vm}(S^{-1}f -P_{\vn} S^{-1}f) \| \leq c n_1^{-\frac{s-1}{2}} + c n_2^{-\frac{s-1}{2}}.
\end{equation*}
\end{prop}
\begin{proof}
By choosing $\vm$ as in \eqref{eq:mchoice}, $A_{\vm,\vn}=\frac{A}{2}$. Thus Lemma \ref{lem:mchoice} yields $\|V_{\vn,\vm}^{-1}\|\leq \frac{2}{A}$. Since $\|P_\vn\|\leq 1$ and $\|S_\vm\|\leq B$, we have
\begin{equation*}
\|V_{\vn,\vm}^{-1}P_{\vn}S_{\vm}(S^{-1}f -P_{\vn} S^{-1}f) \| \leq \frac{2B}{A}\|S^{-1}f -P_{\vn} S^{-1}f\|  .
\end{equation*}
The desired result follows immediately by applying Proposition \ref{prop:projectionerror} on the last term of the above inequality.
\end{proof}

We next estimate $\|V_{\vn,\vm}^{-1}P_{\vn} (S- S_{\vm})S^{-1}f \|$, the final term in \eqref{error-decomposition-inv}.

\begin{prop}\label{prop:thirdterm}
Suppose $\{\psi_\vj\}$ is a localized frame as in \eqref{eq:localizedcondition} and assumption \eqref{eq:smoothassumption} holds. If we choose $\vm$ as in \eqref{eq:mchoice}, then there exists a positive constant $c$ such that
\begin{equation*}
\|V_{\vn,\vm}^{-1}P_{\vn} (S- S_{\vm})S^{-1}f \|  \leq c m_1^{-\frac{s-1}{2}} + c m_2^{-\frac{s-1}{2}}.
\end{equation*}
\end{prop}
\begin{proof}
By Proposition \ref{prop:secondterm}, we have $\|V_{\vn,\vm}^{-1}\|$ is uniformly bounded by $2/A$. It suffices to show
\begin{equation*}
\| (S- S_{\vm})S^{-1}f \|  \leq c m_1^{-\frac{s-1}{2}} + c m_2^{-\frac{s-1}{2}}.
\end{equation*}
Observe that
\begin{equation}\label{eq:thirdtermdecomp}
\| (S- S_{\vm})S^{-1}f \|  \leq \biggl\|\sum_{j_1>m_1}\sum_{j_2\in \bN} \langle S^{-1}f, \psi_\vj\rangle \psi_\vj \biggr\|  + \biggl\|\sum_{j_1\in \bN}\sum_{j_2>m_2} \langle S^{-1}f, \psi_\vj\rangle \psi_\vj \biggr\|.
\end{equation}
We will estimate the first term on the right hand side of the above inequality. The second term would follows in a similar way. Note that $ \langle S^{-1}f, \psi_\vj\rangle =  \langle f,  S^{-1}\psi_\vj\rangle =  \langle f,  \tpsi_\vj\rangle$. It follows that
\begin{equation*}
\biggl\|\sum_{j_1>m_1}\sum_{j_2\in \bN} \langle S^{-1}f, \psi_\vj\rangle \psi_\vj \biggr\|^2 = \sum_{j_1>m_1}\sum_{j_2\in \bN}  \sum_{l_1>m_1}\sum_{l_2\in \bN} \langle f,  \tpsi_\vj\rangle \langle f,  \tpsi_\vl\rangle \langle \psi_\vj, \psi_\vl \rangle.
\end{equation*}
By the localization condition \eqref{eq:localizedcondition} and Lemma \ref{lem:dualcoefficients}, there exists a positive constant $c_1$ such that
\begin{equation*}
\biggl\|\sum_{j_1>m_1}\sum_{j_2\in \bN} \langle S^{-1}f, \psi_\vj\rangle \psi_\vj \biggr\|^2  \leq c_1 \sum_{j_1>m_1}\sum_{j_2\in \bN}  \sum_{l_1>m_1}\sum_{l_2\in \bN} \|\vj\|^{-s} \|\vl\|^{-s} (1+\|\vj - \vl\|_2)^{-s},
\end{equation*}
when $s > 2$.  It follows from the proof of Lemma \ref{lem:projectionerror} that the right hand side of the above inequality is bounded by $c_2 m_1^{-(s-1)}$ for some positive constant $c_2$. That is,
\begin{equation*}
\biggl\|\sum_{j_1>m_1}\sum_{j_2\in \bN} \langle S^{-1}f, \psi_\vj\rangle \psi_\vj \biggr\|^2  \leq c_2 m_1^{-(s-1)}.
\end{equation*}
Similarly, we could obtain
\begin{equation*}
 \biggl\|\sum_{j_1\in \bN}\sum_{j_2>m_2} \langle S^{-1}f, \psi_\vj\rangle \psi_\vj \biggr\|  \leq c_3 m_1^{-(s-1)},
\end{equation*}
for some positive constant $c_3$.
The desired result follows from substituting the above two inequalities into \eqref{eq:thirdtermdecomp}.
\end{proof}

We now summarize all the estimated for the three terms in \eqref{error-decomposition-inv} to obtain an estimate for $\| S^{-1}f - V_{\vn,\vm}^{-1}P_{\vn} f\|$.

\begin{thm}
Suppose $\{\psi_\vj\}$ is a localized frame as in \eqref{eq:localizedcondition} and assumption \eqref{eq:smoothassumption} holds. If we choose $\vm$ as in \eqref{eq:mchoice}, then there exists a positive constant $c$ such that
\begin{equation*}
\| S^{-1}f - V_{\vn,\vm}^{-1}P_{\vn} f\| \leq c n_1^{-\frac{s-1}{2}} + c n_2^{-\frac{s-1}{2}}.
\end{equation*}
\end{thm}
\begin{proof}
It follows immediately by combining the estimates in Propositions \ref{prop:projectionerror}, \ref{prop:secondterm}, and \ref{prop:thirdterm}.
\end{proof}


\section{Approximating $S^{-1}$ for weakly localized frames in two dimensions}\label{sec:admissibleframes}
We now discuss the case when the sampling frame, $\{\psi_\vj: \vj\in \bN^2\}$, is only {\it weakly localized}, that is, the condition \eqref{eq:localizedcondition} holds for some $0<s\leq 2$. We point out that Wiener's Lemma employed in the proof of Lemma \ref{lem:dualcoefficients} will not hold and the dual frame coefficients may not have a similar decay as the frame coefficients. Thus the convergence analysis techniques of the Casazza-Christensen method in Section \ref{sec:localizedframes} cannot be used. In particular, the Casazza-Christensen method converges very slowly when the sampling frame is only weakly localized. In \cite{Song2011b} the technique of {\em admissible frames} was introduced to project a weakly localized frame onto a more localized frame which enabled a better approximation to the inverse frame operator for one-dimensional problems.  We demonstrate below how this method can be extended to include two-dimensional frames. To this end, we first define admissibility for two-dimensional frames.

 \begin{defn}\label{def:2dadmissible}
 A frame $\{\phi_{\vj}: \vj \in \bN^2\}$ is {\em admissible} with respect to a frame $\{\psi_{\vj}: \vj \in \bN^2\}$ if
 \begin{enumerate}[(1)]\label{eq:admlocalized}
  \item It is intrinsically self-localized
  \begin{equation}
   \left|\langle \phi_{\vj}, \phi_{\vl} \rangle\right| \leq \gamma_0 (1+ \|\vj-\vl\|_2)^{-t}, \quad \gamma_0>0, \, \vj, \vl\in \bN^2
  \end{equation}
with a localization rate $t>2$, that is, $\{\phi_\vj: \vj \in \bN^2\}$ satisfies Definition \ref{def:2dlocalized}, and
  \item The inner product of the two frames is bounded by
  \begin{equation}\label{eq:admissibleframe}
   \left|\langle \psi_{\vj}, \phi_{\vl} \rangle\right| \leq \gamma_1 (1+ |j_1-l_1|)^{-s}(1+ |j_2-l_2|)^{-s}, \quad  \gamma_1>0,\, s>0,\, \vj, \vl\in \bN^2.
  \end{equation}
 \end{enumerate}
\end{defn}

As in \cite{Song2011b}, the idea here is to use the projection onto a finite-dimensional subspace spanned by the admissible frame $\phi_\vj$ to approximate the dual frame of the weakly localized frame $\psi_\vj$. For any $\vn=(n_1, n_2)\in \bN^2$, we let $\cG_{\vn}:=\spn\{\phi_{\vj}: \vone\leq \vj\leq \vn\}$ be the finite-dimensional subspace of $\cH$ and $Q_\vn$ be the orthogonal projection from $\cH$ onto $\cG_\vn$. Observe that $\{\phi_{\vj}: \vone\leq \vj\leq \vn\}$ is a frame for $\cG_\vn$ (c.f. \cite{Christensen2003}). For $\vm, \vn\in \bN^2$, let $W_{\vn,\vm}$ be the restriction of $Q_\vn S_\vm$ to $\cG_\vn$:
\begin{equation}
\label{eq:W_nm}
W_{\vn,\vm}:= Q_\vn S_\vm |_{\cG_\vn},
\end{equation}
where $S_\vm$ is the finite frame operator associated with the original frame $\{\psi_{\vj}: \vj \in \bN^2\}$ given in (\ref{finiteframe}). As stated previously, there always exists an $\vm = \vm(\vn)$ such that an approximation to $S^{-1}$ exists and converges, \cite{Casazza2000}.  Below we demonstrate that the operator $W_{\vn,\vm(\vn)}$ is invertible, and that moreover we can approximate $S^{-1}$ for weakly localized frames using (\ref{eq:W_nm}) with
\begin{equation}
\|S^{-1}f - W_{\vn, \vm(\vn)}^{-1}Q_\vn f \| \rightarrow 0, \quad \mbox{ as } \vn\rightarrow \infty.
\label{eq:weak_inv}
 \end{equation}
We further find an explicit relationship for $\vm = \vm(\vn)$. For ease of presentation we subsequently write $W_{\vn}$ for the operator $W_{\vn,\vm(\vn)}$ since both subscripts depend on $\vn$.

In order to approximate the convergence rate of (\ref{eq:weak_inv}), we consider the following symbolic decomposition
\begin{equation*}
S^{-1}f - W_{\vn}^{-1}Q_{\vn} f = S^{-1}(f - Q_{\vn}f) + S^{-1}(W_{\vn} - S)W_{\vn}^{-1}Q_{\vn}f
\end{equation*}

We use the lower frame bound $A$ to give a bound of $S^{-1}$ and simplify the approximation error as follows
\begin{equation}\label{adm-error-decomp}
 \|S^{-1}f - W_{\vn}^{-1}Q_{\vn} f\| \leq \frac{1}{A}\|f - Q_{\vn}f\| + \frac{1}{A}\|(W_{\vn} - S)W_{\vn}^{-1}Q_{\vn}f\|
\end{equation}

As in Section \ref{sec:localizedframes}, we assume that $f$ satisfies the smoothness assumption (\ref{eq:smoothassumption}) for the admissible frame $\{\phi_{\vj}: \vj \in \bN^2\}$.  This allows us to study the error introduced by the projections onto $\cG_{\vn}.$ Lemma \ref{lem:projectionerror} immediately provides an estimate for the first term on the right hand side of \eqref{adm-error-decomp} given by:
\begin{prop}\label{prop:admprojerror}
If $\{\phi_\vj\}$ is a localized frame as in \eqref{eq:admlocalized} satisfying \eqref{eq:smoothassumption}, then there exists a positive constant $c$ such that for any $\vn\in \bN^2$
\begin{equation*}
\|f- Q_\vn f\|\leq c n_1^{-\frac{t-1}{2}} + c n_2^{-\frac{t-1}{2}}.
\end{equation*}
\end{prop}

To analyze the second term, we will show that $\|W_\vn f - Sf\|$ converges to zero for $f \in \cG_{\vn}$ and that $W_{\vn}^{-1}$ exists and is uniformly bounded for all ${\vn}$. We first define the following constant
\begin{equation*}
B_{\vm, \vn} = \frac{4s \gamma_1^2}{ (2s-1)^2 \lambda_{\min}(\mPhi_\vn)} n_1 n_2 \bigl[ (m_1 -n_1)^{-(2s-1)} +  (m_2 -n_2)^{-(2s-1)} \bigr],
\end{equation*}
where $\mPhi_\vn =\bigl[ \langle \phi_\vj, \phi_\vl \rangle \bigr]_{\vone\leq \vj,\vl\leq \vn}$ and $\lambda_{\min}(\mPhi_\vn)$ is its smallest eigenvalue. We assume the matrix $\mPhi_\vn$ is invertible. Otherwise, we can use its invertible principle submatrix instead.  As in Section \ref{sec:localizedframes}, we always choose $m_1, m_2$ larger than $n_1, n_2$ respectively.

\begin{lem}\label{lem:Wnconverge}
For $m_1\geq n_1$ and $m_2\geq n_2$, it holds that $\|S-W_{\vn}\| \leq B_{\vm,\vn}.$
\end{lem}
\begin{proof}
Since $S-W_{\vn}$ and $Q_{\vn}$ are self adjoint,
\begin{equation}\label{eq:Wnerror}
 \|S - W_{\vn}\|_{\cG_{\vn}} = \sup_{g\in\cG_{\vn}} \frac{ \bigl| \langle(S-W_{\vn})g,g\rangle \bigr|}{\langle g, g\rangle} = \sup_{g\in\cG_{\vn}} \frac{\bigl| \langle Sg,g\rangle - \langle S_{\vm}g,g\rangle \bigr|}{\langle g, g\rangle}.
\end{equation}
We proceed by finding an upper bound for $\langle Sg,g\rangle - \langle S_{\vm}g,g\rangle$. Expanding the frame operator yields
$
\langle Sg,g\rangle = \bigl\langle \sum_{\vj\in \bN^{2}} \langle g, \psi_{\vj}\rangle \psi_{\vj},g  \bigr\rangle = \sum_{\vj\in \bN^{2}}|\langle g, \psi_{\vj}\rangle|^{2}.
$
We can find a similar expression using the partial frame operator $S_{\vm}$. Subtracting the two expressions results in the upper bound,
\begin{equation*}
\langle Sg,g\rangle - \langle S_{\vm}g,g\rangle = \sum_{\vj \in \bN^{2}}|\langle g, \psi_{\vj}\rangle|^{2} - \sum_{\vone \leq \vj \leq \vm}|\langle g, \psi_{\vj}\rangle|^{2} = \sum_{j_1 > m_1 \mbox{ or } j_2>m_2}|\langle g, \psi_\vj \rangle|^2.
\end{equation*}
The remaining analysis follows closely to that of the proof of Lemma \ref{lem:mchoice}. Observe that
\begin{equation}\label{eq:tail2}
\sum_{j_1 > m_1 \mbox{ or } j_2>m_2}|\langle g, \psi_\vj \rangle|^2 \leq \sum_{j_1>m_1}\sum_{j_2\in \bN} |\langle g, \psi_\vj \rangle|^2 + \sum_{j_1\in \bN}\sum_{j_2>m_2} |\langle g, \psi_\vj \rangle|^2.
\end{equation}
We will obtain an upper bound for the first term on the right hand side of the above inequality. An upper bound for the second term could be obtained similarly. Since $g\in \cG_\vn$, we can write $g=\sum_{\vone\leq\vl\leq \vn} a_\vl \phi_\vl$ for some $\va=(a_\vl: \vone\leq\vl\leq \vn)$. It follows that
\begin{equation*}
\sum_{j_1>m_1}\sum_{j_2\in \bN} |\langle g, \psi_\vj \rangle|^2 = \sum_{j_1>m_1}\sum_{j_2\in \bN} \biggl|\sum_{\vone\leq\vl\leq \vn}a_\vl \langle \phi_\vl, \psi_\vj \rangle \biggr|^2.
\end{equation*}
By Cauchy-Schwartz inequality, we have
\begin{equation*}
\sum_{j_1>m_1}\sum_{j_2\in \bN} |\langle g, \psi_\vj \rangle|^2 \leq \|\va\|_2^2 \sum_{j_1>m_1}\sum_{j_2\in \bN} \sum_{\vone\leq\vl\leq \vn} \bigl|\langle \phi_\vl, \psi_\vj \rangle \bigr|^2.
\end{equation*}
Applying the admissible frame condition \eqref{eq:admissibleframe} to the above inequality yields
\begin{equation}\label{eq:taildecomp2}
\sum_{j_1>m_1}\sum_{j_2\in \bN} |\langle g, \psi_\vj \rangle|^2 \leq \gamma_1^2 \|\va\|_2^2 \sum_{j_1>m_1}\sum_{l_1=1}^{n_1}  (1+|j_1-l_1|)^{-2s} \sum_{j_2\in \bN}\sum_{l_2=1}^{n_2}  (1+|j_2-l_2|)^{-2s}.
\end{equation}
When $s>\frac{1}{2}$, for any $l_2\in \bN$, we have $\sum_{j_2\in \bN} (1+|j_2-l_2|)^{-2s} \leq 2\sum_{t\in \bN} t^{-2s} \leq \frac{4s}{2s-1}$. That is,
\begin{equation*}
\sum_{j_2\in \bN}\sum_{l_2=1}^{n_2} (1+|j_2-l_2|)^{-2s} \leq \frac{4 s n_2}{2s-1}.
\end{equation*}
Moreover, it follows from the proof of Lemma 3.4 in \cite{Song2011b} that
\begin{equation*}
\sum_{j_1>m_1}\sum_{l_1=1}^{n_1} (1+|j_1-l_1|)^{-2s} \leq \frac{n_1}{2s-1} (m_1-n_1)^{-(2s-1)}.
\end{equation*}
Substituting the above two inequalities into \eqref{eq:taildecomp2}, we have
\begin{equation*}
\sum_{j_1>m_1}\sum_{j_2\in \bN} |\langle g, \psi_\vj \rangle|^2 \leq \|\va\|_2^2 \frac{4s \gamma_1^2}{ (2s-1)^2} n_1 n_2 (m_1 -n_1)^{-(2s-1)}.
\end{equation*}
Similarly, we could obtain
\begin{equation*}
\sum_{j_1\in \bN}\sum_{j_2>m_2} |\langle g, \psi_\vj \rangle|^2 \leq \|\va\|_2^2 \frac{4s \gamma_1^2}{ (2s-1)^2} n_1 n_2 (m_2 -n_2)^{-(2s-1)}.
\end{equation*}
Substituting the above two inequalities into \eqref{eq:tail2} yields
\begin{equation*}
\sum_{j_1 > m_1 \mbox{ or } j_2>m_2}|\langle g, \psi_\vj \rangle|^2 \leq \|\va\|_2^2 \frac{4s \gamma_1^2}{ (2s-1)^2} n_1 n_2 \bigl[ (m_1 -n_1)^{-(2s-1)} +  (m_2 -n_2)^{-(2s-1)} \bigr].
\end{equation*}
We next find a lower bound for the denominator in \eqref{eq:Wnerror}:
\begin{equation*}
\langle g, g \rangle = \biggl\|\sum_{\vone\leq\vl\leq \vn} a_\vl \phi_\vl \biggr\|^2 = \va^T \mPhi_\vn \va \geq \lambda_{\min}(\mPhi_\vn) \|\va\|_2^2.
\end{equation*}
Combining the above two inequalities with \eqref{eq:Wnerror}, we have
\begin{equation*}
 \|S - W_{\vn}\|_{\cG_{\vn}} \leq  \frac{4s \gamma_1^2}{(2s-1)^2\lambda_{\min}(\mPhi_\vn)} n_1 n_2 \bigl[ (m_1 -n_1)^{-(2s-1)} +  (m_2 -n_2)^{-(2s-1)} \bigr],
\end{equation*}
which concludes the proof.
\end{proof}

We next present a condition under which $W_{\vn,\vm}$ is invertible.

\begin{lem}\label{lem:admmchoice}
If $B_{\vm, \vn} < A,$ then the set $\{Q_{\vn}\psi_{\vj}: \vone\leq \vj\leq \vm\}$ forms a frame for $\cG_{\vn}$ with frame bounds $A - B_{\vm,\vn}$ and $B$, and the corresponding frame operator is $W_{\vm,\vn}.$ Moreover,
\begin{equation*}
\|W^{-1}_{\vn,\vm}\|\leq \frac{1}{A -B_{\vm,\vn}}.
\end{equation*}
\end{lem}
\begin{proof}
We will show $\{Q_{\vn}\psi_{\vj}: \vone\leq \vj\leq \vm\}$ forms a frame for $\cG_{\vn}$ by checking the frame conditions directly. We first observe that for $g\in\cG_{\vn}, \langle g, Q_{\vn} \psi_{\vj}\rangle =  \langle Q_{\vn}g,  \psi_{\vj}\rangle = \langle g,  \psi_{\vj}\rangle,$ where the first equality follows from the fact that $Q_{\vn}$ is self-adjoint. An upper bound then follows directly from the frame condition for $\{\psi_{\vj}\}_{\vj\in \bN^2}$. Indeed, for any $\vm \in \bN^2,$
\begin{equation*}
\sum_{\vone\leq \vj\leq \vm} |\langle g, Q_\vn \psi_\vj \rangle|^2 = \sum_{\vone\leq \vj\leq \vm} |\langle g, \psi_\vj \rangle|^2 \leq  \sum_{\vj\in \bN^2} |\langle g, \psi_\vj \rangle|^2 \leq B\|g\|^2.
\end{equation*}
The lower bound relies on the inequality $$\sum_{j_1 > m_1 \mbox{ or } j_2>m_2}|\langle g, \psi_\vj \rangle|^2 \leq B_{\vm,\vn}\|g\|^2$$ from Lemma \ref{lem:Wnconverge} and on the lower frame bound for $\{\psi_{\vj}\}_{\vj\in \bN^2}.$ Specifically,
\begin{equation*}
\sum_{\vone\leq \vj\leq \vm} |\langle g, Q_\vn \psi_\vj \rangle|^2 = \sum_{\vj\in \bN^2} |\langle g, \psi_\vj \rangle|^2 - \sum_{j_1 > m_1 \mbox{ or } j_2>m_2}|\langle g, \psi_\vj \rangle|^2  \geq (A - B_{\vm,\vn})\|g\|^2,
\end{equation*}
which implies the frame conditions are satisfied for $\{Q_{\vn}\psi_{\vj}: \vone\leq \vj\leq \vm\}$.

It remains to show that the corresponding frame operator is $W_{\vm,\vn}.$ For $g \in \cG_{\vn},$
\begin{equation*}
\sum_{\vone\leq \vj\leq \vm} \langle g, Q_\vn \psi_\vj\rangle Q_\vn \psi_\vj  = \sum_{\vone\leq \vj\leq \vm} \langle g,  \psi_\vj\rangle Q_\vn \psi_\vj = Q_\vn S_\vm g =W_{\vn, \vm}g.
\end{equation*}
Thus  $W_{\vn, \vm}$ is the associated frame operator and is therefore invertible. The bound on the inverse frame operator is the reciprocal of the lower frame bound.
\end{proof}

We now present an estimate of $\|(W_{\vn} - S)W_{\vn}^{-1}Q_{\vn}f\|$, the second term on the right hand side of \eqref{adm-error-decomp}, by giving a specific choice of $\vm$ depending on $\vn$.

\begin{prop}\label{prop:admsecondterm}
Suppose $\{\phi_\vj\}$ is an admissible frame for $\{\psi_\vj\}$   as in \eqref{eq:admissibleframe} and assumption \eqref{eq:smoothassumption} holds. If
\begin{equation}\label{eq:admmchoice}
m_j = n_j +  \left(\frac{8s \gamma^2 n_1 n_2}{A (2s-1)^2 \lambda_{\min}(\mPsi_\vn)}\right)^{\frac{1}{2s-1}}n_j^{\frac{t-1}{2(2s-1)}}, \quad j=1,2,
\end{equation}
then $\|W_{\vn, \vm}^{-1}\| \leq \frac{2}{A}\left(n_1^{-\frac{t-1}{2}} + n_2^{-\frac{t-1}{2}}\right)$ and there exists a positive constant $c$ such that
\begin{equation*}
\|(W_{\vn} - S)W_{\vn}^{-1}Q_{\vn}\| \leq c n_1^{-\frac{t-1}{2}} + c n_2^{-\frac{t-1}{2}}.
\end{equation*}
\end{prop}
\begin{proof}
We point out that $B_{\vm,\vn}=\frac{A}{2}\left(n_1^{-\frac{t-1}{2}} + n_2^{-\frac{t-1}{2}}\right)$ with such a choice of $\vm$. By Lemmas \ref{lem:Wnconverge} and \ref{lem:admmchoice}, it follows that
\begin{equation*}
\|(W_{\vn} - S)W_{\vn}^{-1}Q_{\vn}f\| \leq \|W_{\vn} - S \| \|W_{\vn,\vm}^{-1}\| \|Q_{\vn}\| \|f\| \leq  \frac{B_{\vm,\vn}}{A-B_{\vm,\vn}}\|f\| .
\end{equation*}
The desired result follows by plugging in $B_{\vm,\vn}$.
\end{proof}

We now present the main result of this section.

\begin{thm}
Suppose $\{\psi_\vj\}$ is a localized frame as in \eqref{eq:localizedcondition} and assumption \eqref{eq:smoothassumption} holds. If we choose $\vm$ as in \eqref{eq:admmchoice}, then there exists a positive constant $c$ such that
\begin{equation*}
\| S^{-1}f - W_{\vn,\vm}^{-1}Q_{\vn} f\| \leq c n_1^{-\frac{t-1}{2}} + c n_2^{-\frac{t-1}{2}}.
\end{equation*}
\end{thm}
\begin{proof}
It follows immediately by combining the estimates in Propositions \ref{prop:admprojerror} and \ref{prop:admsecondterm} into the error decomposition \eqref{adm-error-decomp}.
\end{proof}


\section{Applications of Fourier frames and numerical implementation}\label{sec:numerical}

In this section we present a method for approximating an unknown function $f$ from its finite sampling data $\vf=\{\langle f, \psi_\vj\rangle: \vone\leq \vj \leq \vm\}$ using the admissible frame approach for approximating the inverse frame operator $S^{-1}$. In particular, we define $f_{\vn,\vm}$ as the approximation of $f$ given by
\begin{equation*}
f_{\vn,\vm}:= \sum_{\vone\leq \vj \leq \vm}\langle f, \psi_\vj\rangle W_{\vn,\vm}^{-1} Q_{\vn} f.
\end{equation*}
Following the one-dimensional construction derived in \cite{Song2011b}, we have
\begin{equation}
f_{\vn,\vm} = \sum_{\vone\leq \vl\leq \vn} c_\vl \phi_\vl, 
\label{eq:fnAF}
\end{equation}
where $\vc= \mOmega^\dagger \vf$, $\mOmega=[\langle \psi_\vj, \phi_\vl\rangle]_{\vone\leq \vj\leq \vm, \vone\leq \vl\leq \vn},$ and $\mOmega^\dagger$ is the Moore-Penrose pseudo-inverse of $\mOmega$.  We refer to (\ref{eq:fnAF}) as the admissible frame approximation of $f$.

For our numerical experiments, we consider the sampling frame (Definition \ref{def:frame}) given by
\begin{equation*}
\psi_\vj = \se^{-\pi i \vlambda_\vj},
\end{equation*}
where each $\lambda_\vj$ lies on some non-uniform sampling patterns in 2D. The patterns in our examples reflect prototypes of what may arise in a variety of applications, including magnetic resonance imaging (MRI) and synthetic aperture radar (SAR).  We note that in several applications the sampling patterns do not in fact yield a frame,  except for its given finite span.  However, our numerical algorithm provides a stable and robust mechanism for reconstructing an image in the projected admissible frame space.  As in  \cite{GelbSong2014}, we anticipate that our method can be sped up numerically by using the NFFT. This will be considered in future work. Figure \ref{fig:allpatterns} displays the sampling patterns described below.
\begin{figure}[h!]
\centering
\subfigure[Jittered]{\includegraphics[width=.22\textwidth]{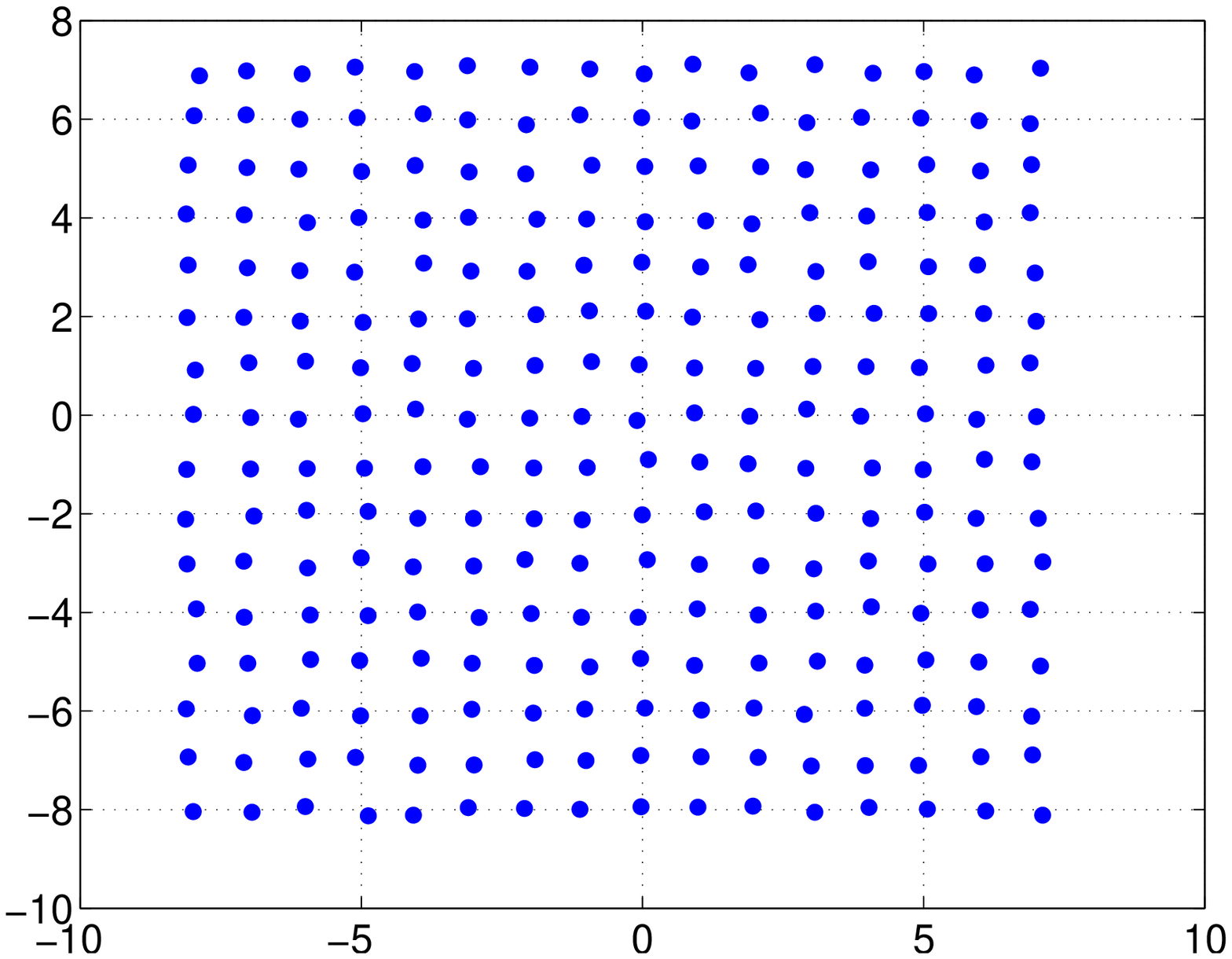}}
\subfigure[Rosette]{\includegraphics[width=.22\textwidth]{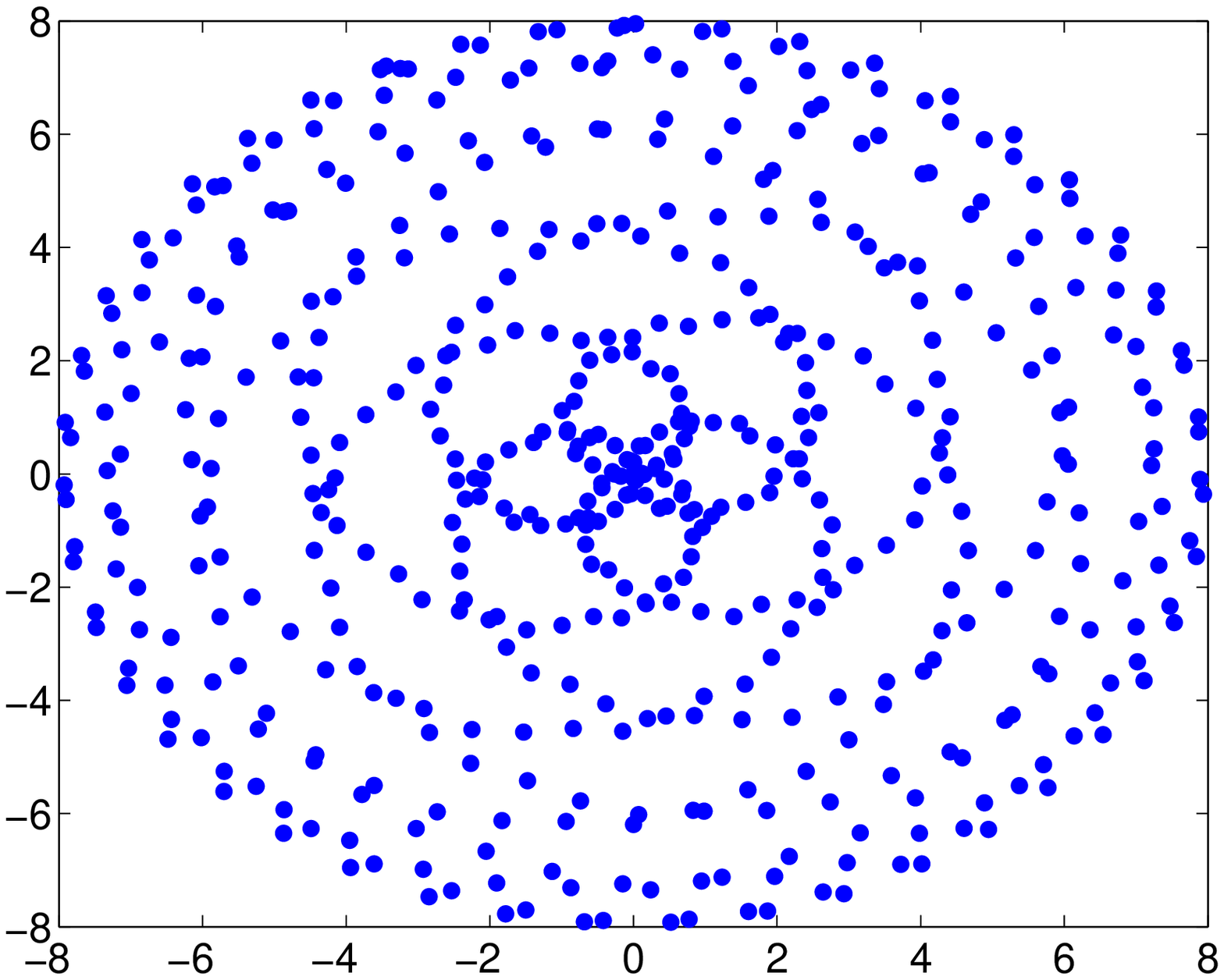}}
\subfigure[Spiral]{\includegraphics[width=.22\textwidth]{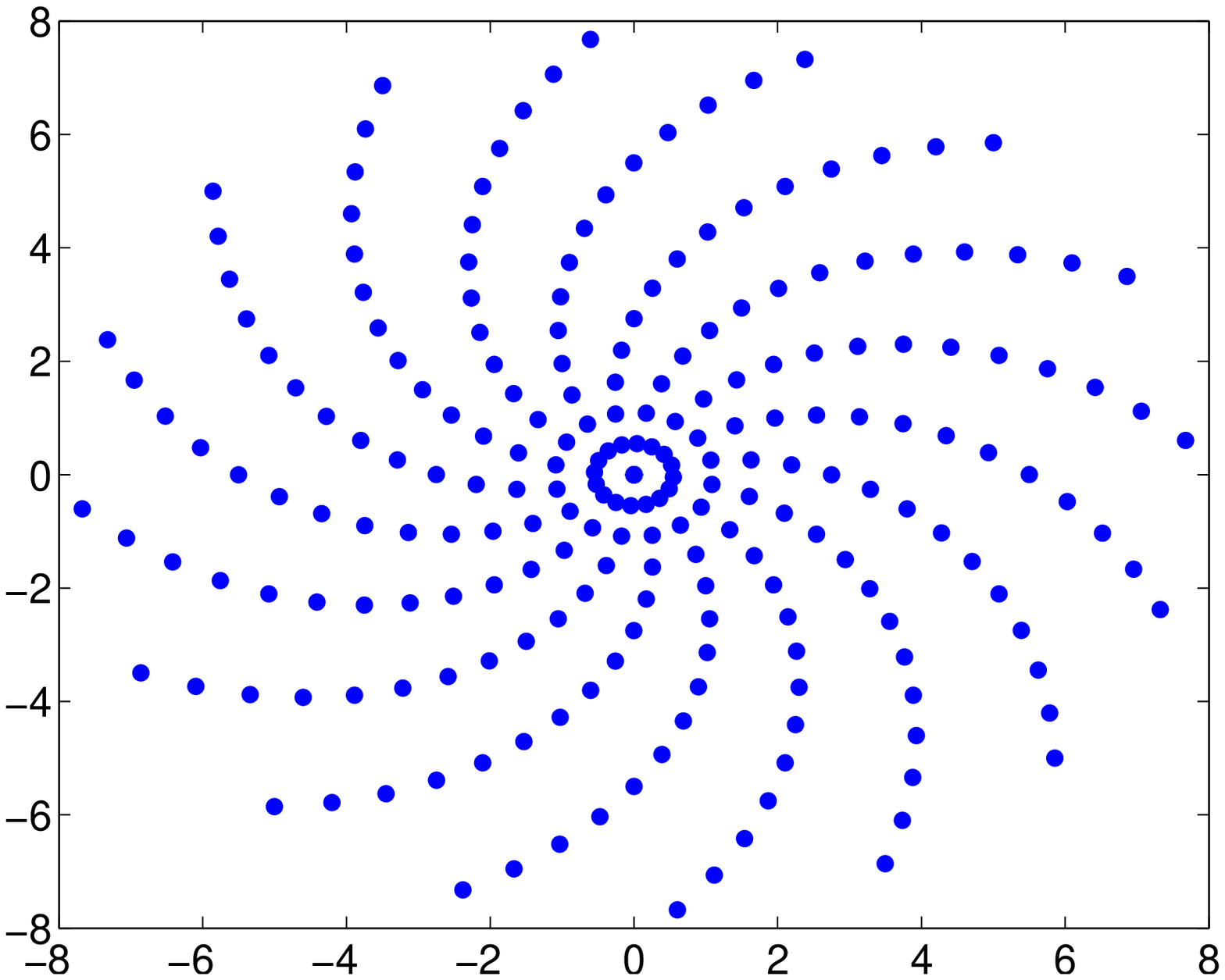}}
\subfigure[Polar]{\includegraphics[width=.22\textwidth]{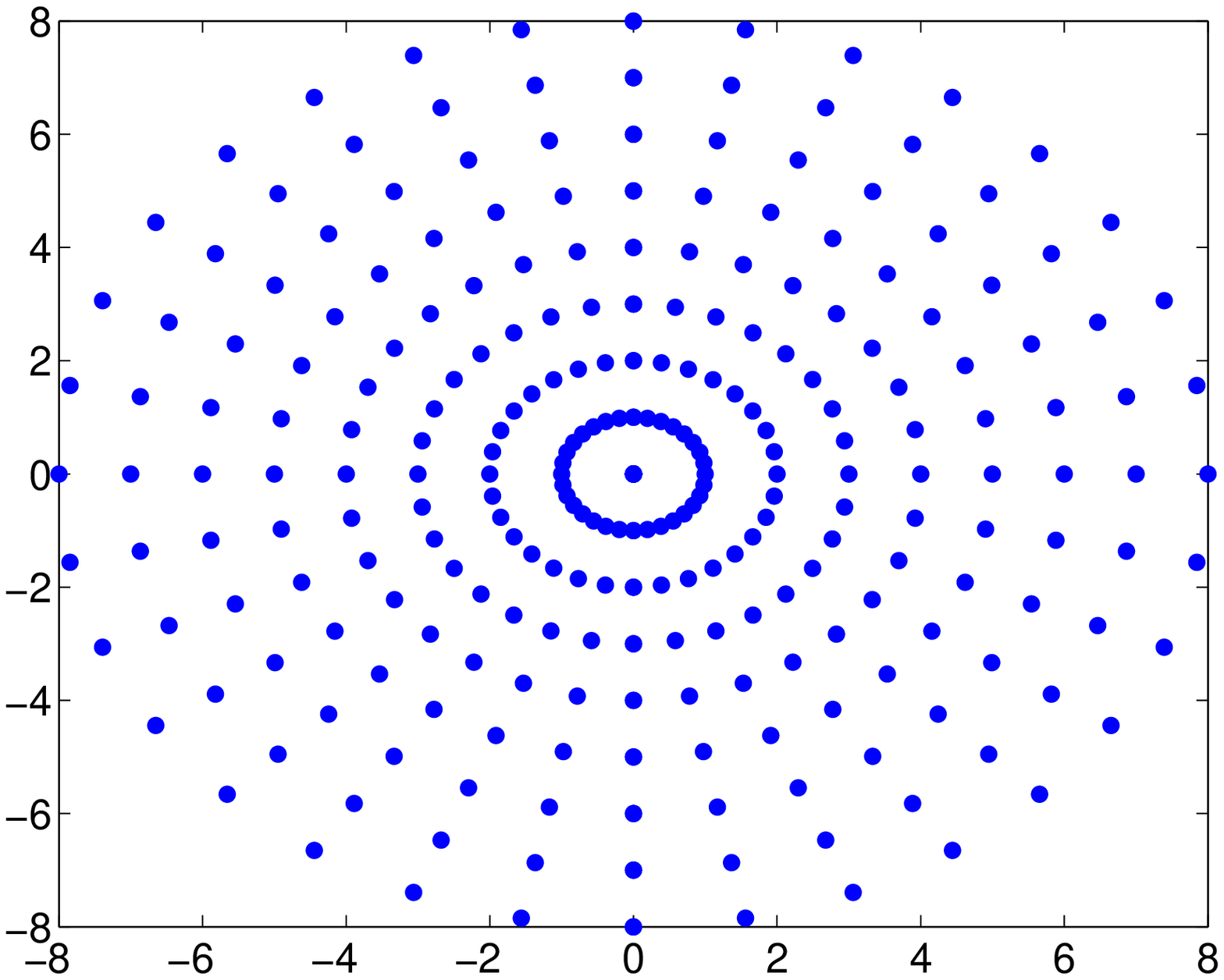}}
\caption{Examples of various sampling patterns}
\label{fig:allpatterns}
\end{figure}
\begin{enumerate}[(1)]
\item Jittered sampling:
\begin{equation*}
\lambda_\vj = \vj + \vepsilon_\vj,
\end{equation*}
where $\vepsilon_\vj$ is a small perturbation in 2D. Typically, we assume $\vepsilon_\vj\in [-1/4, 1/4]^2$. Note that the jittered sampling case constitutes a frame in $L^2[-1,1]$ when the perturbations lie in such a range \cite{Sun1999159}.

\item Rosette sampling:
\begin{equation*}
\lambda_j = k_{max} (\cos(w_1 t_j)\cos(w_2t_j), \cos(w_1 t_j)\sin(w_2t_j)),
\end{equation*}
where $k_{max}, w_1, w_2$ are positive constants and $t_j\in [0, T]$ for some $T>0$.

\item Spiral sampling:
\begin{equation*}
\lambda_\vj = (c\theta_\vj \cos(2\pi \theta_\vj),  c\theta_\vj\sin(2\pi \theta_\vj)),
\end{equation*}
where $c>0$ and $\theta_\vj>0$.

\item Polar sampling:
\begin{equation*}
\lambda_\vj = (c r_{j_1} \cos(\theta_{j_2}), c r_{j_1} \cos(\theta_{j_2})),
\end{equation*}
where $c>0$, $r_{j_1}=\frac{j_1}{R}\in [-1/2, 1/2)$, and $\theta_{j_2} = \frac{\pi j_2}{T} \in [-\pi/2, \pi/2)$.

\end{enumerate}

In what follows, we compare our admissible frame (A-F) method, (\ref{eq:fnAF}),  with the Casazza-Christensen (C-C) method, which uses (\ref{eq:V_nm}) in the approximation of (\ref{eq:f_frame}). It is important to note that the main difference between the C-C method and the A-F method is that in the C-C method, the sampling frame and the reconstruction frame are the same, while the A-F method chooses an admissible frame for the reconstruction based on the given sampling frame.  Here we use the standard Fourier basis as the reconstruction basis for all sampling patterns, which can be shown to be admissible with respect to jittered sampling using an extension of the one-dimensional proof of admissibility shown in \cite{Song2011b}. We remark that the other sampling patterns (rosette, spiral, and polar) may fail to satisfy the admissibility condition with respect to the standard Fourier basis. However, our numerical experiments demonstrate that our method still out-performs the C-C algorithm  even though the admissibility condition fails to hold.  Hence we conclude that our method is robust with respect to different sampling patterns.  We start by considering the following test function:

\begin{example}
\label{ex:example1}
\begin{equation*}
f_1(\vx)=\sin(4\pi x_1) \sin(2\pi x_2), \quad \vx=(x_1, x_2)\in [-1,1]^2.
\end{equation*}
\end{example}

\begin{table}[htbp]
\centering
\small
\begin{tabular}{c| c c| c c| c c| c c}
	&\multicolumn{2}{|c|}{Jittered} &\multicolumn{2}{|c|}{Rosette} & \multicolumn{2}{|c|}{Spiral} & \multicolumn{2}{c}{Polar} \\
$M$ & A-F & C-C& A-F & C-C& A-F & C-C& A-F & C-C\\
\hline
$8^2$ &2.5E-1&2.5E-1 &	2.7E0&3.1E-1& 1.9E-1 & NaN &9.1E-1&NaN\\
$16^2$&2.4E-14& 6.0E-3&1.4E-1&3.7E-1& 2.2E-2&NaN& 1.1E-2&NaN\\
$32^2$&2.0E-15&1.3E-3&7.0E-4&8.4E-2&1.2E-5&NaN&1.2E-5&NaN	\\
$64^2$& 1.9E-15&7.4E-4&1.9E-5&4.2E-4&1.2E-9& NaN&4.2E-5& 3.4E-2\\
\bottomrule
\end{tabular}
 \caption{Mean square errors using the A-F and C-C methods for Example \ref{ex:example1}.}
\label{table:mse}
\end{table}

Table \ref{table:mse} compares the mean square errors for approximating $f_1(\vx)$ in Example \ref{ex:example1} using A-F and C-C methods for various data size $M=m_1m_2$. In both cases, the number of reconstruction frame elements, $N = n_1n_2$, is chosen according to Proposition \ref{prop:admsecondterm}, given by (\ref{eq:admmchoice}).  Figures \ref{fig:image_jit} and \ref{fig:image_other} show the reconstructed images when $M = 64^2$ using both methods for jittered, rosette, spiral, and polar sampling, respectively.  A one-dimensional cross section of the reconstructed functions when $f_1 = 0$ is shown in Figure \ref{fig:slice} for each respective pattern. Since the C-C method fails to approximate $f_1(\vx)$ when the data are sampled on the spiral pattern, no value is given in Figure \ref{fig:slice}(c).

\begin{figure}[htbp]
\centering
        \subfigure[$f_1(\vx)$]{\includegraphics[scale = .21]{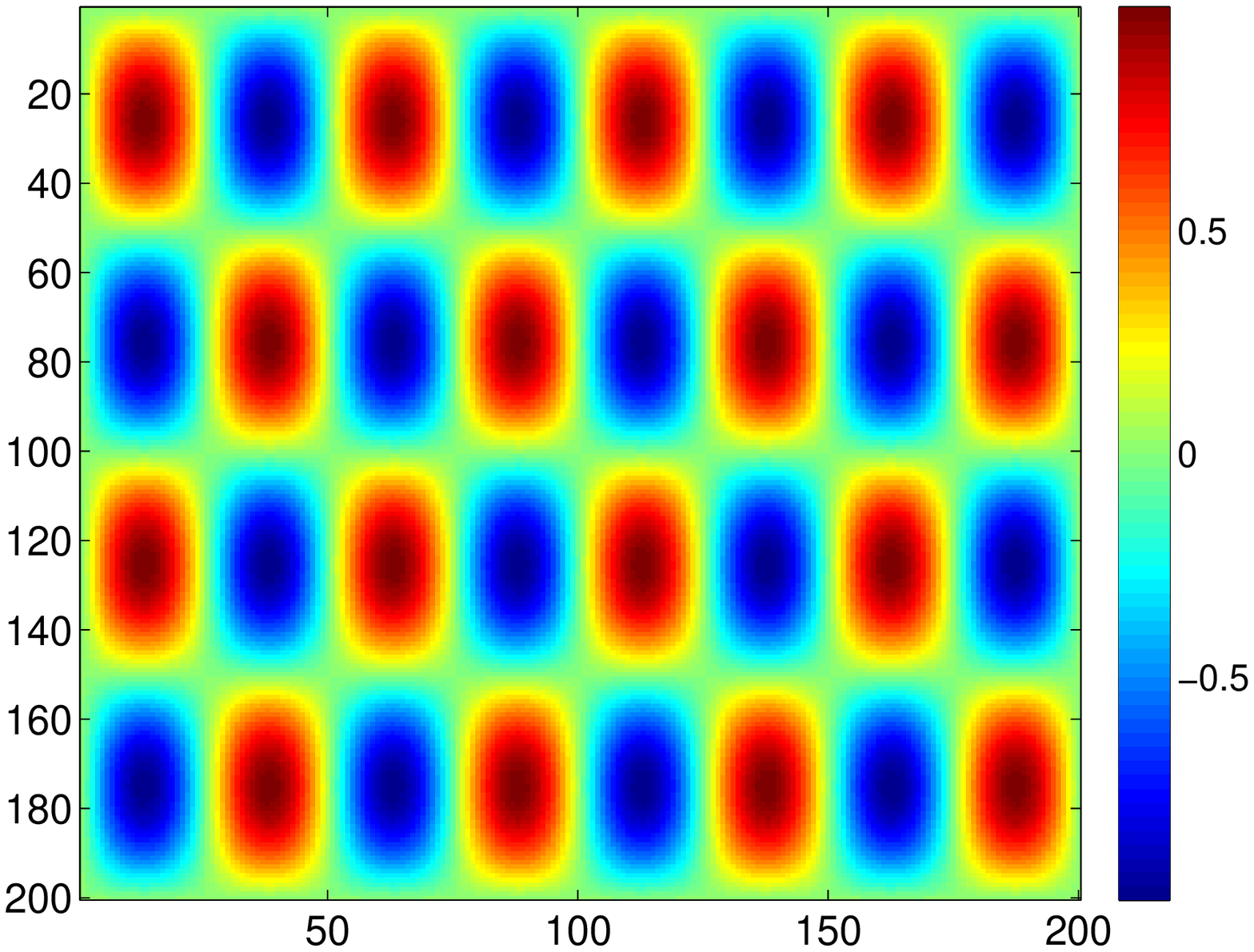}}
        \subfigure[A-F reconstruction]{\includegraphics[scale = .21]{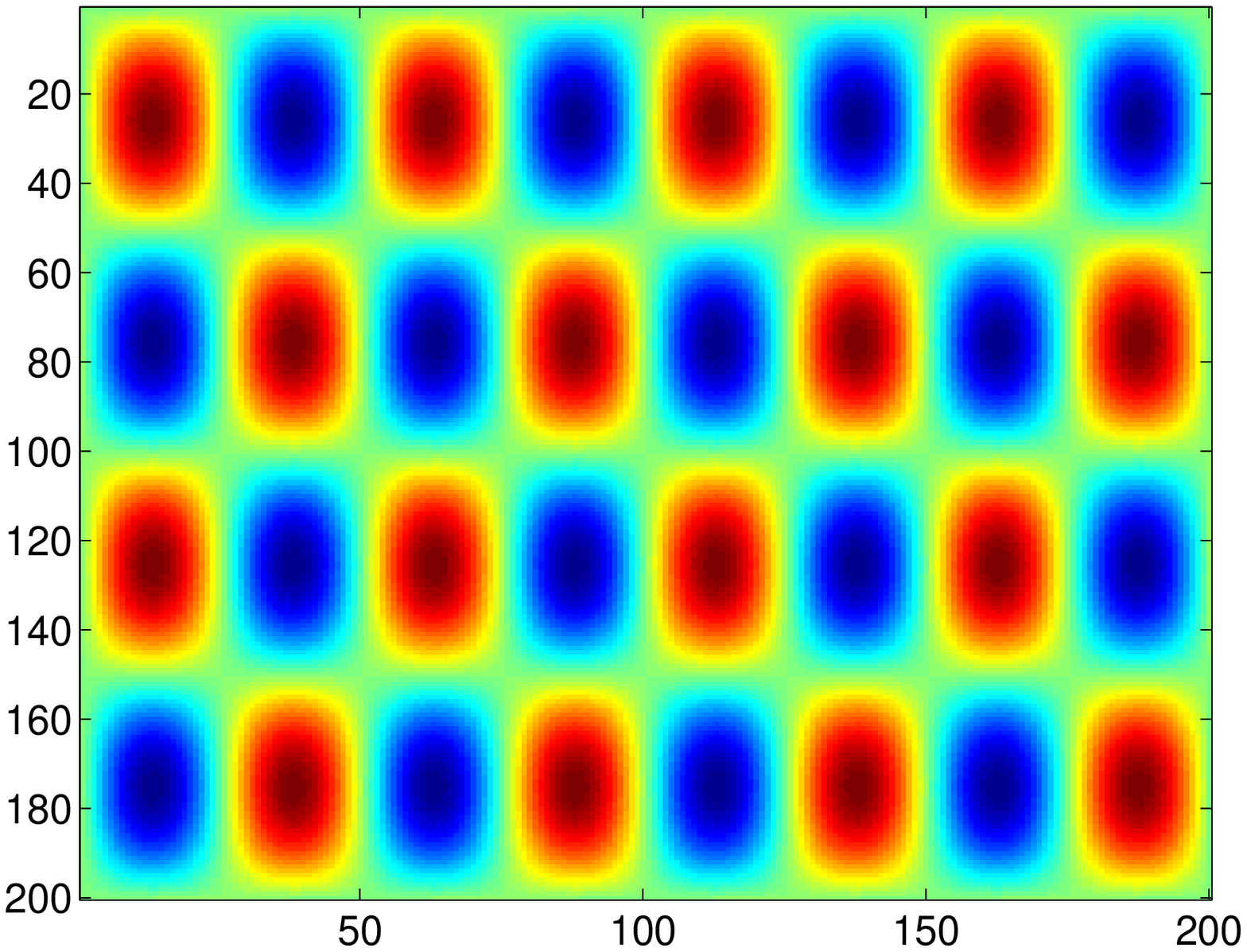}}
 	  \subfigure[C-C reconstruction]{\includegraphics[scale = .21]{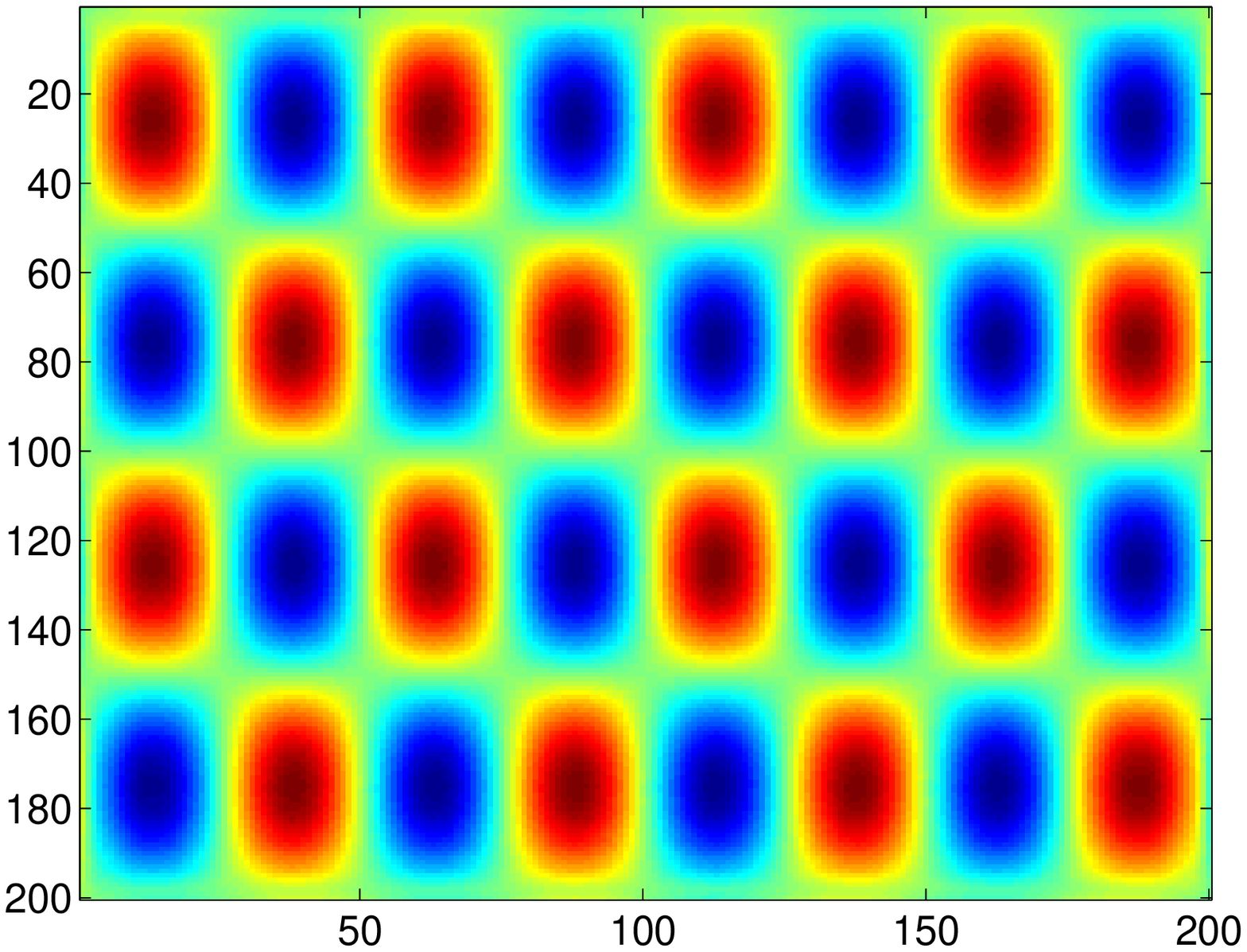}}
        \caption{Reconstruction of $f_1(\vx)$ in Example \ref{ex:example1} with jittered sampling.}
\label{fig:image_jit}
\end{figure}

\begin{figure}[htbp]
\centering
\begin{tabular}{c c c}
Rosette & Spiral & Polar\\
\includegraphics[scale = .21]{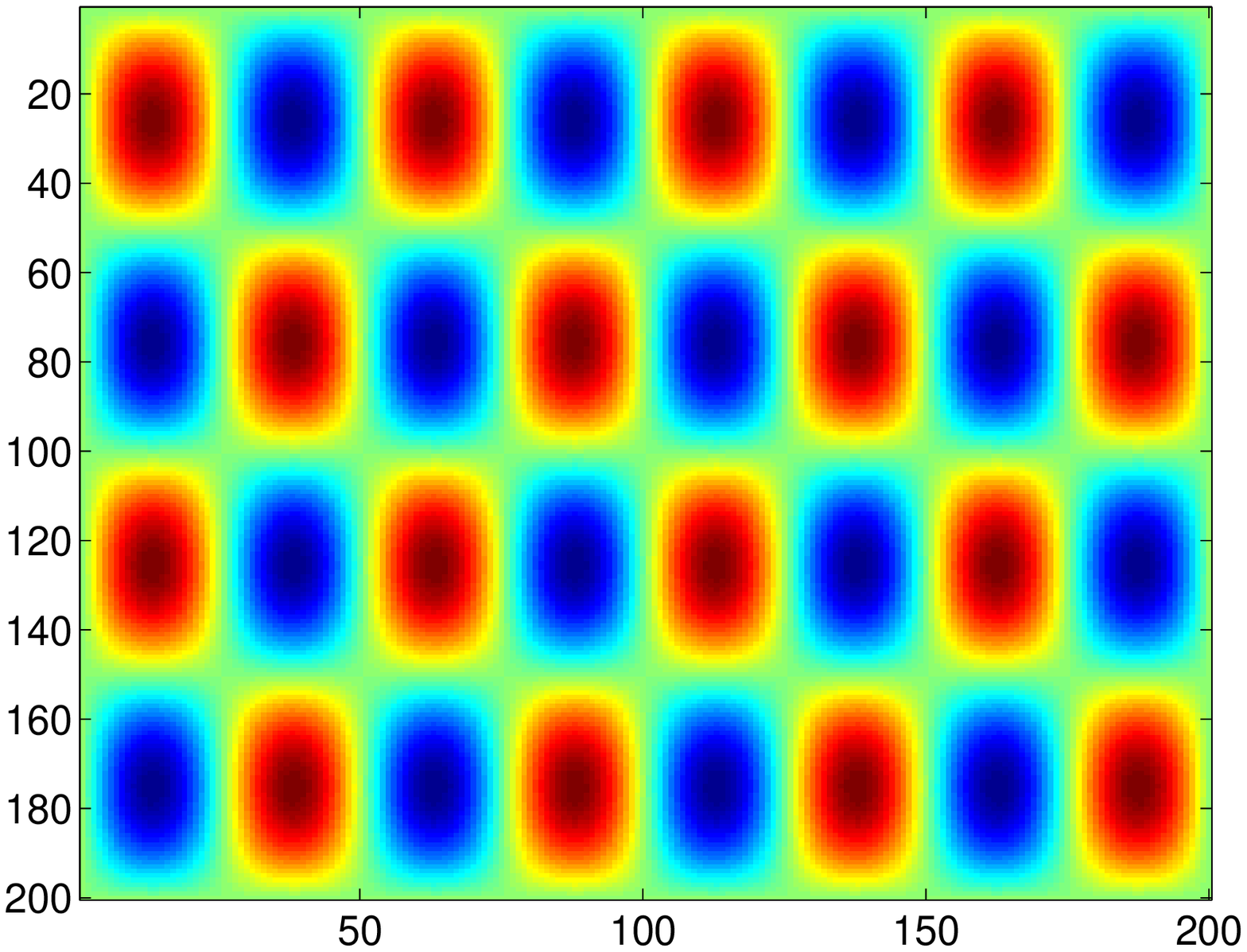} & \includegraphics[scale = .21]{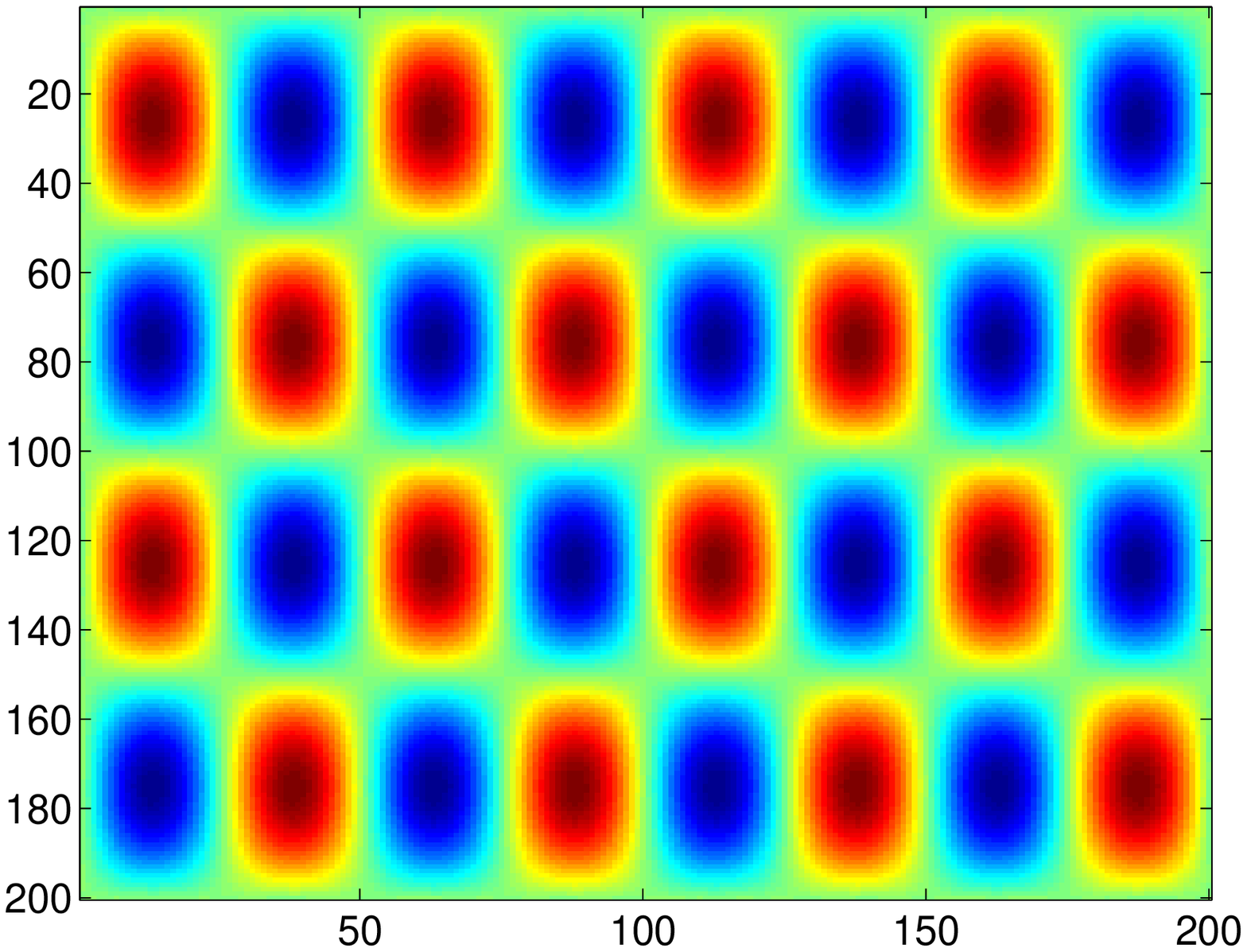} & \includegraphics[scale = .21]{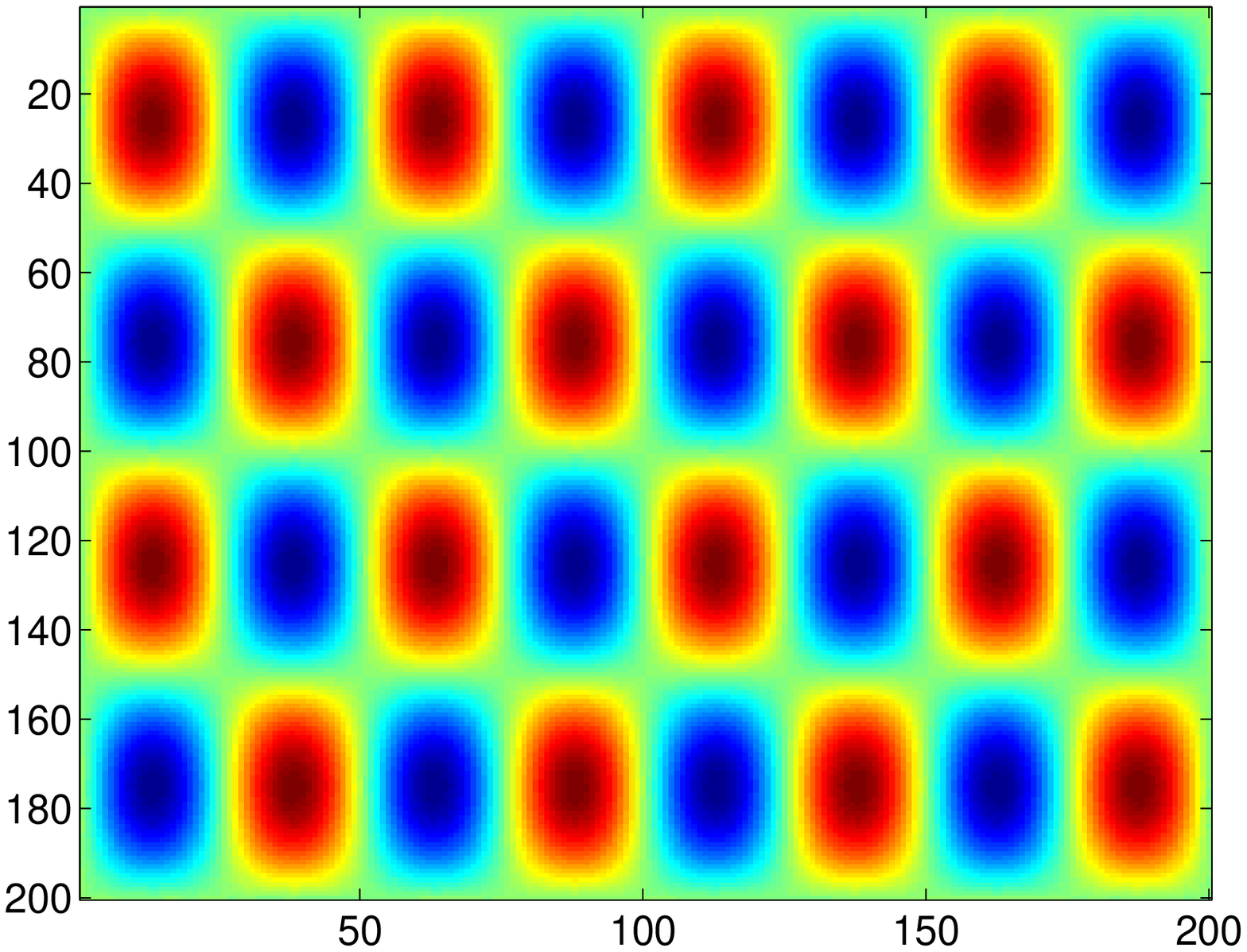}\\
\includegraphics[scale = .21]{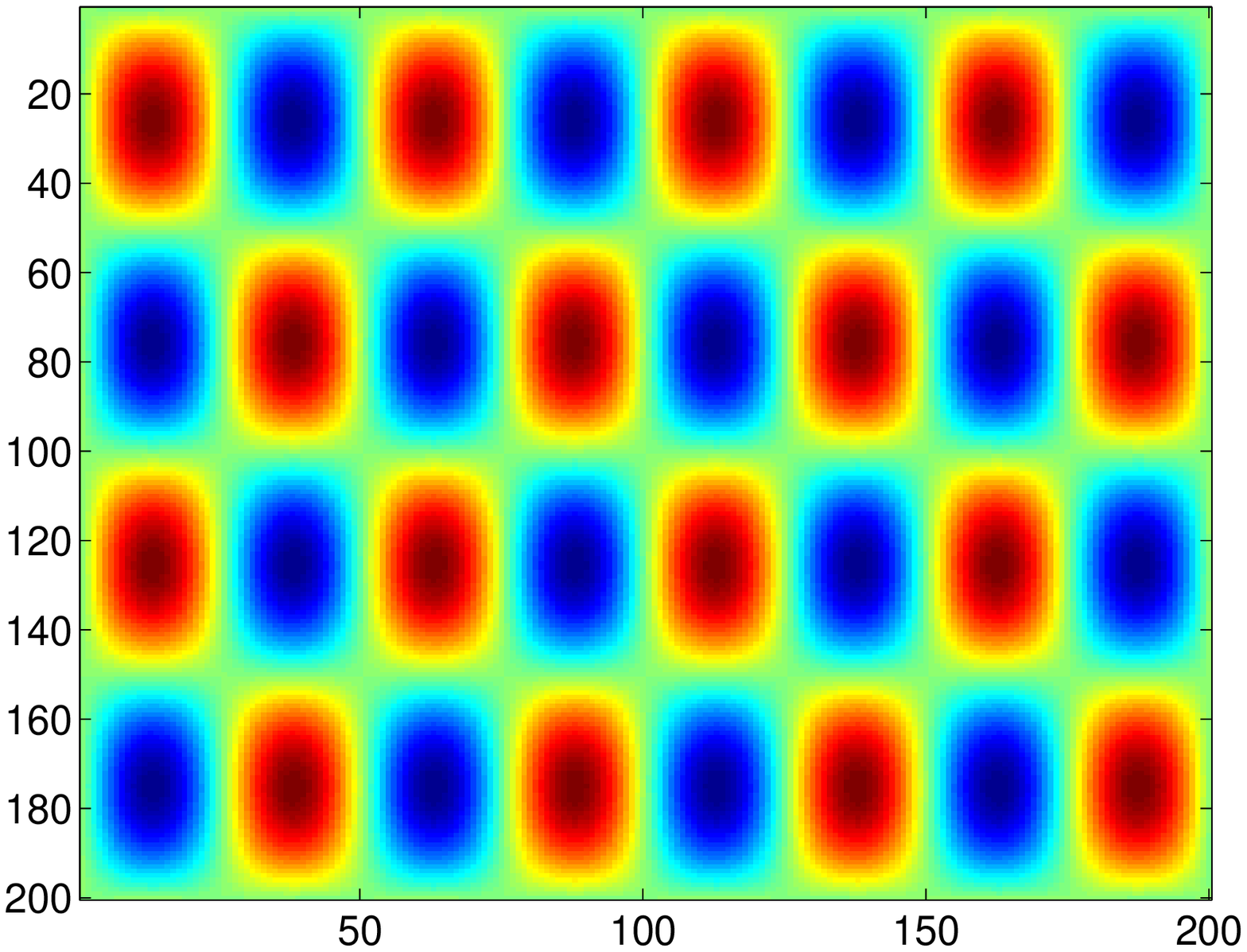} & \includegraphics[scale = .21]{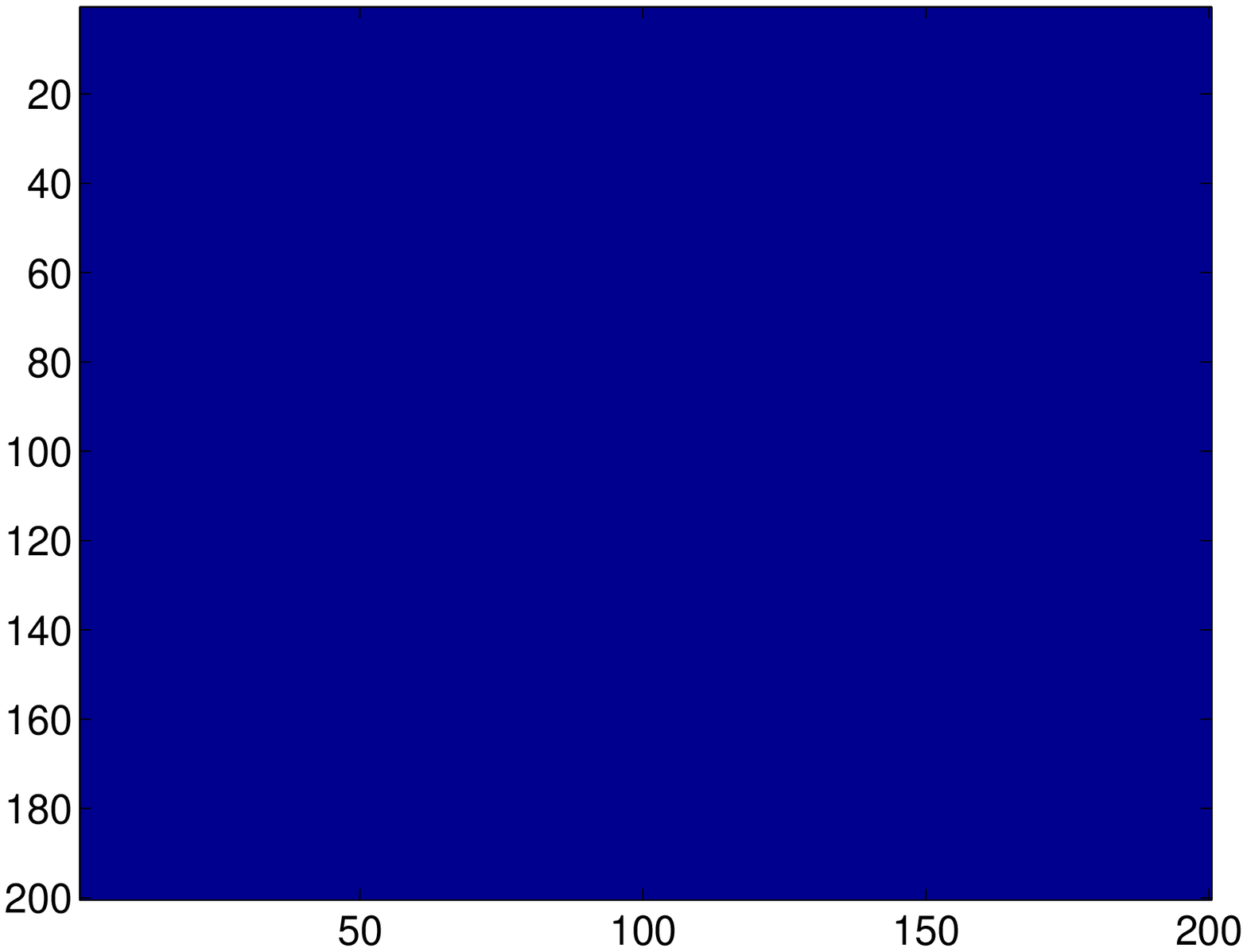}& \includegraphics[scale = .21]{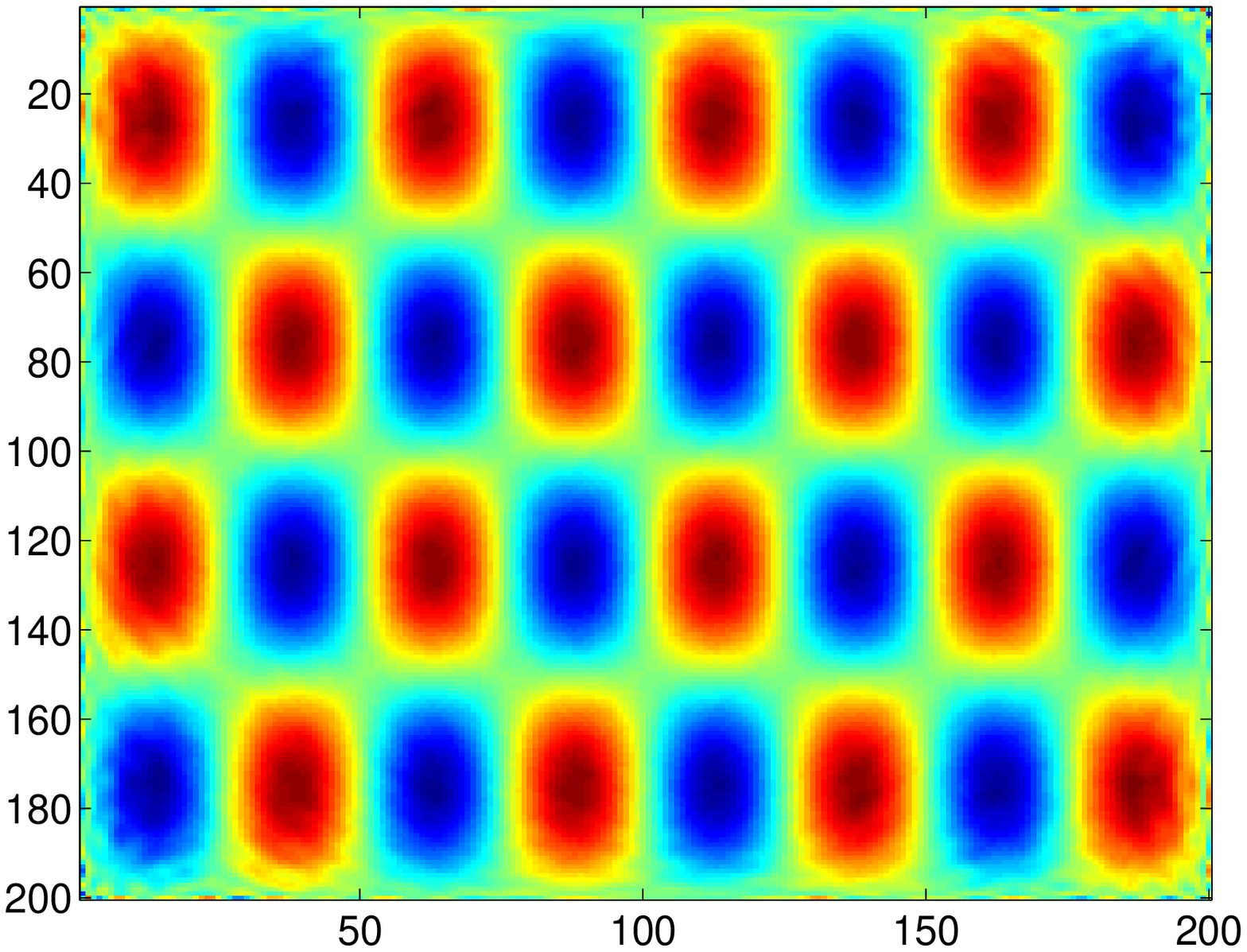}
\end{tabular}
\caption{Reconstruction of $f_1(\vx)$ in Example \ref{ex:example1} with the A-F method (top row) and the C-C method (bottom row).}
\label{fig:image_other}
\end{figure}

\begin{figure}[htbp]
\centering
        \subfigure[jittered sampling]{\includegraphics[scale = .15]{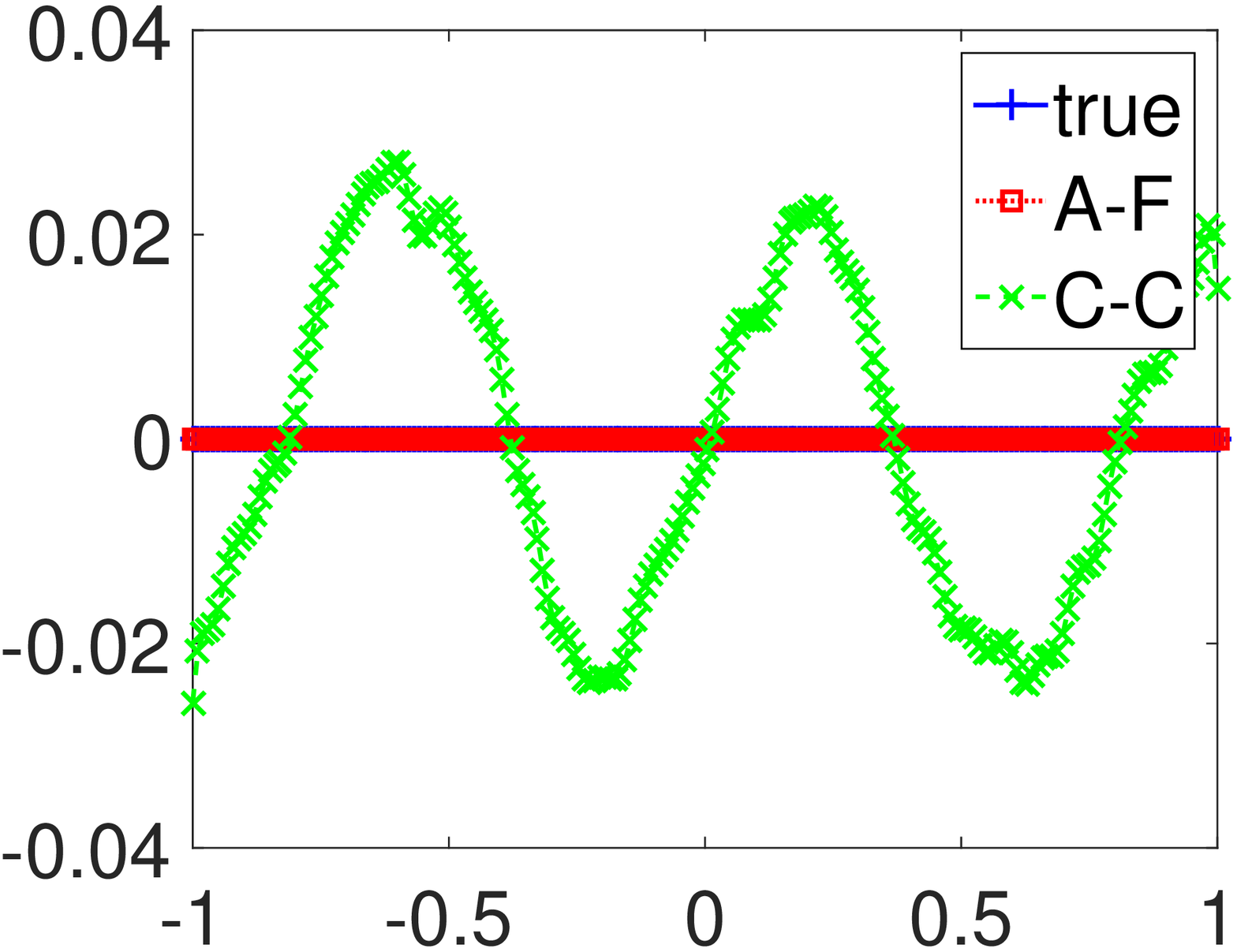}}
        \subfigure[rosette sampling]{\includegraphics[scale = .15]{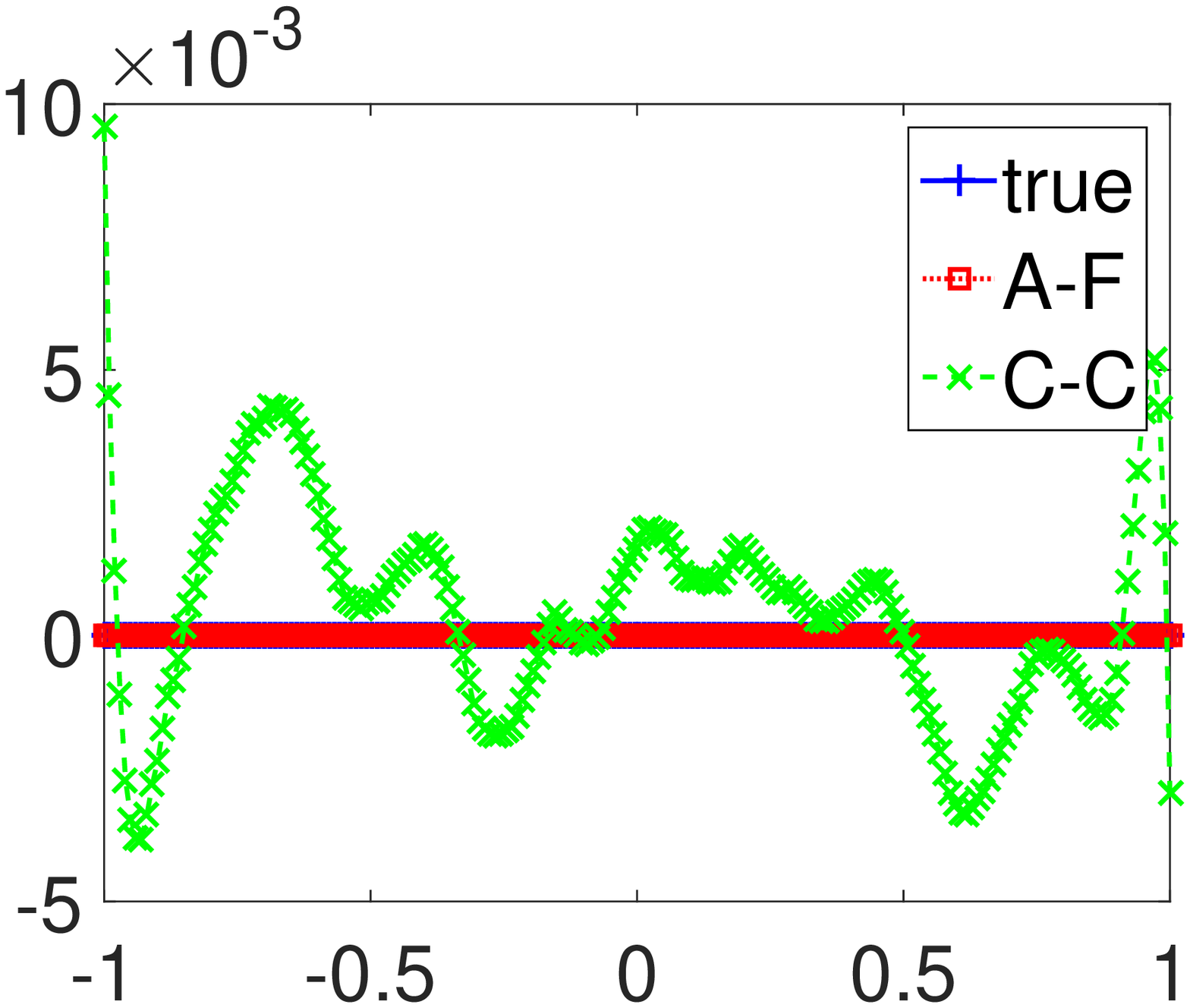}}
        \subfigure[spiral sampling]{\includegraphics[scale = .15]{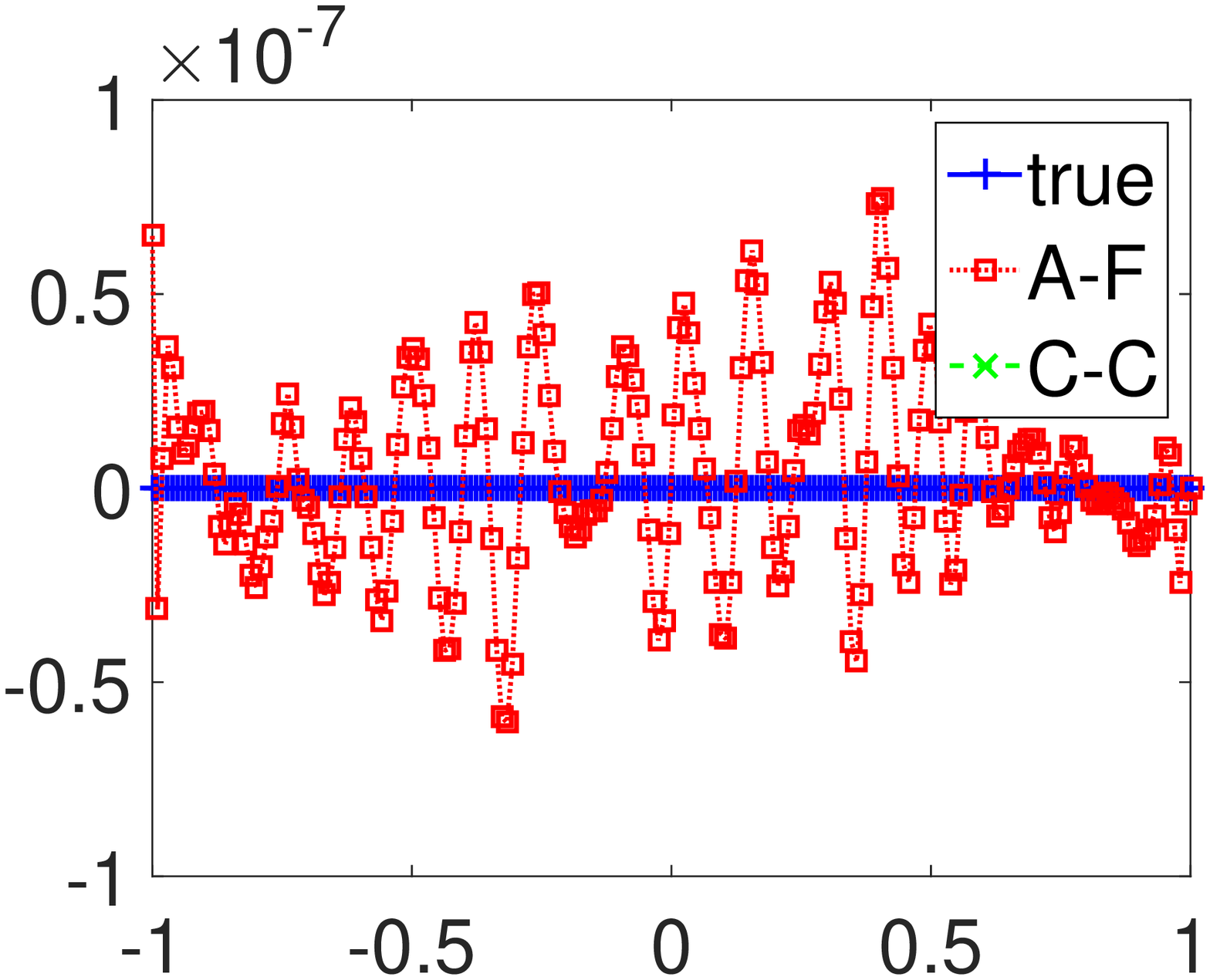}}
        \subfigure[polar sampling]{\includegraphics[scale = .15]{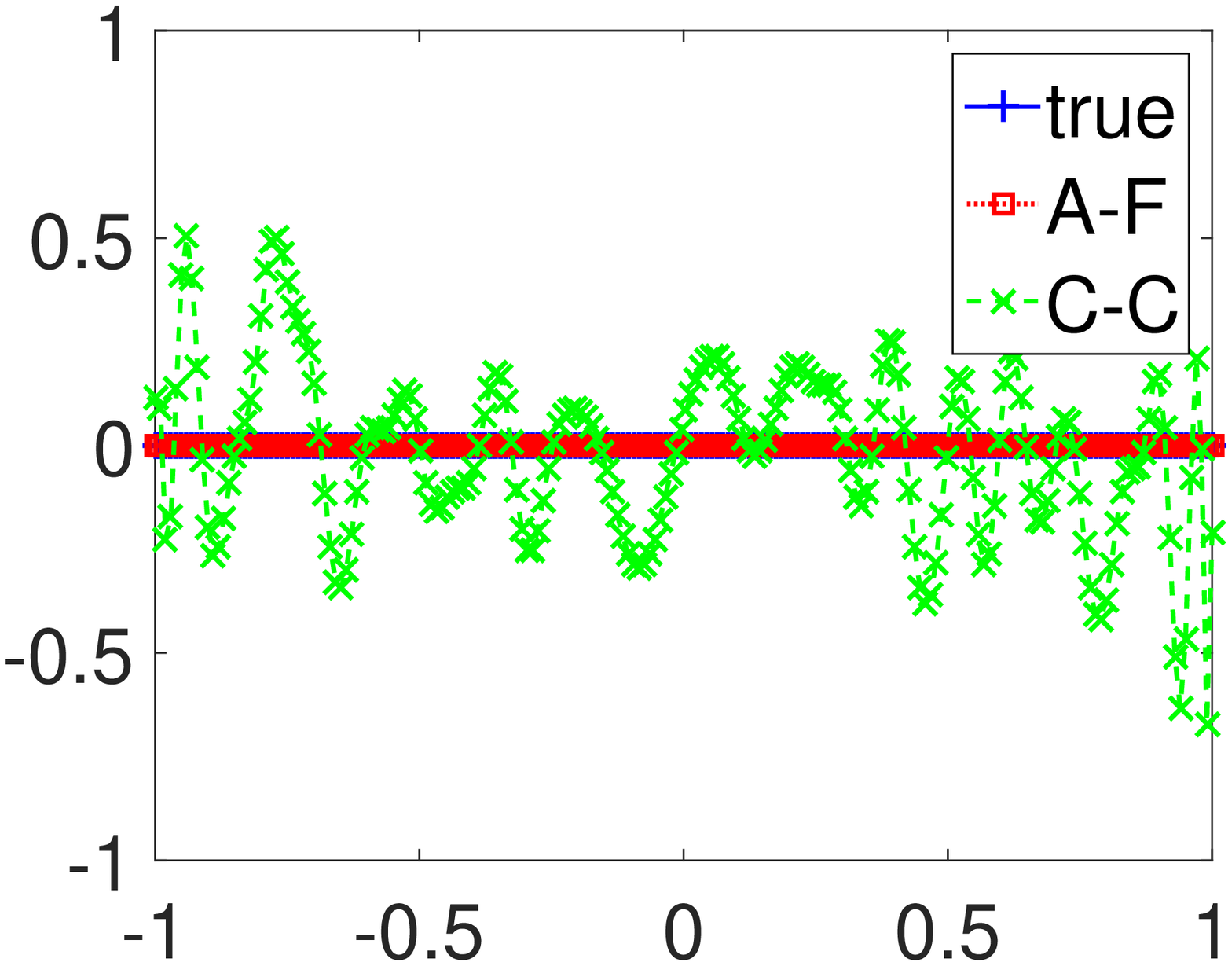}}
\caption{One-dimensional cross section where $f_1(\vx) = 0$ in Example \ref{ex:example1} for all sampling patterns.}
\label{fig:slice}
\end{figure}

For a second experiment, we consider the following function:
\begin{example}
\label{ex:example2}
\begin{equation*}
f_2(\vx) = \sin(4\pi x_1) (x_2^2-1)^2, \quad \vx=(x_1, x_2)\in [-1,1]^2.
\end{equation*}
\end{example}
Table \ref{table:f2mse} compares the mean square errors for approximating $f_2(\vx)$ in Example \ref{ex:example2} using A-F and C-C methods for various data size $M=m_1m_2$. In both cases, the number of reconstruction frame elements, $N = n_1n_2$, is chosen according to Proposition \ref{prop:admsecondterm}, given by (\ref{eq:admmchoice}).  Figures \ref{fig:f2image_jit} and \ref{fig:f2image_other} show the reconstructed images when $M = 64^2$ using both methods for jittered, rosette, spiral, and polar sampling, respectively.  A one-dimensional cross section of the reconstructed functions when $f_2 = 0$ is shown in Figure \ref{fig:f2slice} for each respective pattern. Since the C-C method fails to approximate $f_2(\vx)$ when the data are sampled on the spiral pattern, no value is given in Figure \ref{fig:f2slice}(c). 

\begin{table}[htbp]
\centering
\small
\begin{tabular}{c| c c| c c| c c| c c}
	&\multicolumn{2}{|c|}{Jittered} &\multicolumn{2}{|c|}{Rosette} & \multicolumn{2}{|c|}{Spiral} & \multicolumn{2}{c}{Polar} \\
$M$ & A-F & C-C& A-F & C-C& A-F & C-C& A-F & C-C\\
\hline
$8^2$ &2.2E-1&2.2E-1&9.6E-1&2.1E-1& 4.0E-1 & NaN &4.4E-1&NaN\\
$16^2$&6.2E-5& 7.1E-2&7.4E-2&1.7E-1& 1.7E-1&NaN& 4.4E-1&NaN\\
$32^2$&5.8E-6&3.0E-2&1.2E-3& 7.4E-2&9.1E-3&NaN&1.3E-1& NaN\\
$64^2$& 5.5E-7&1.3E-2&7.1E-4&3.7E1&4.4E-5& NaN&7.8E-2& 3.6E8\\
\bottomrule
\end{tabular}

\caption{Mean square errors using the A-F and C-C methods for Example \ref{ex:example2}.}
\label{table:f2mse}
\end{table}

Due to periodicity,  the Fourier basis is a good reconstruction basis for Example \ref{ex:example1}, and consequently the A-F method can fully resolve $f_1(\vx)$ when the sampling frame consists of jittered samples (and is indeed a frame).  Although the C-C method converges, it does so very slowly.  Since $f_2(\vx)$ does not retain periodicity in its derivatives, the Fourier basis does not provide the ideal reconstruction basis.  Nevertheless, the A-F method still yields fast convergence.   Moreover, even when the sampling pattern does not constitute a frame, the A-F method still provides a convergent reconstruction method for both examples.  This is not true for the C-C method, however. Finally, we note that if the target function is only piecewise smooth, the A-F approach allows the use of high order post-processing methods to improve the overall accuracy, \cite{GelbHines}.  

Our numerical experiments indicate that it may be possible to relax the admissible frame conditions in the convergence analysis of the A-F method.  This is clearly not the case for the C-C method, where the frame condition is essential for convergence. The robustness to sampling patterns of the A-F method is highly desirable as a variety of sampling patterns naturally arise in a number of applications.  Hence future investigations will include more general admissibility conditions that include a variety of sampling patterns.

\begin{figure}[htbp]
\centering
        \subfigure[$f_2(\vx)$]{\includegraphics[scale = .28]{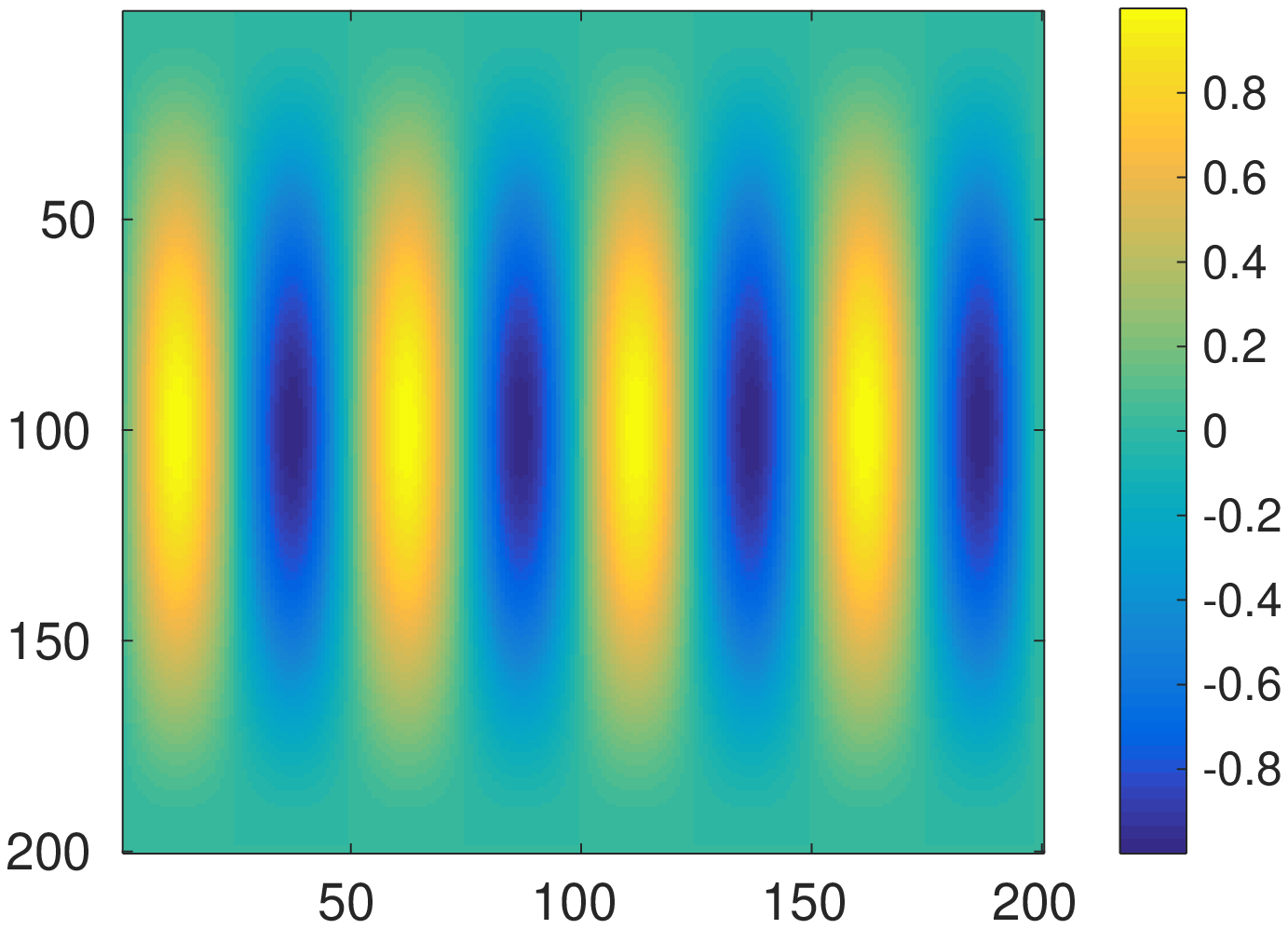}}
        \subfigure[A-F reconstruction]{\includegraphics[scale = .28]{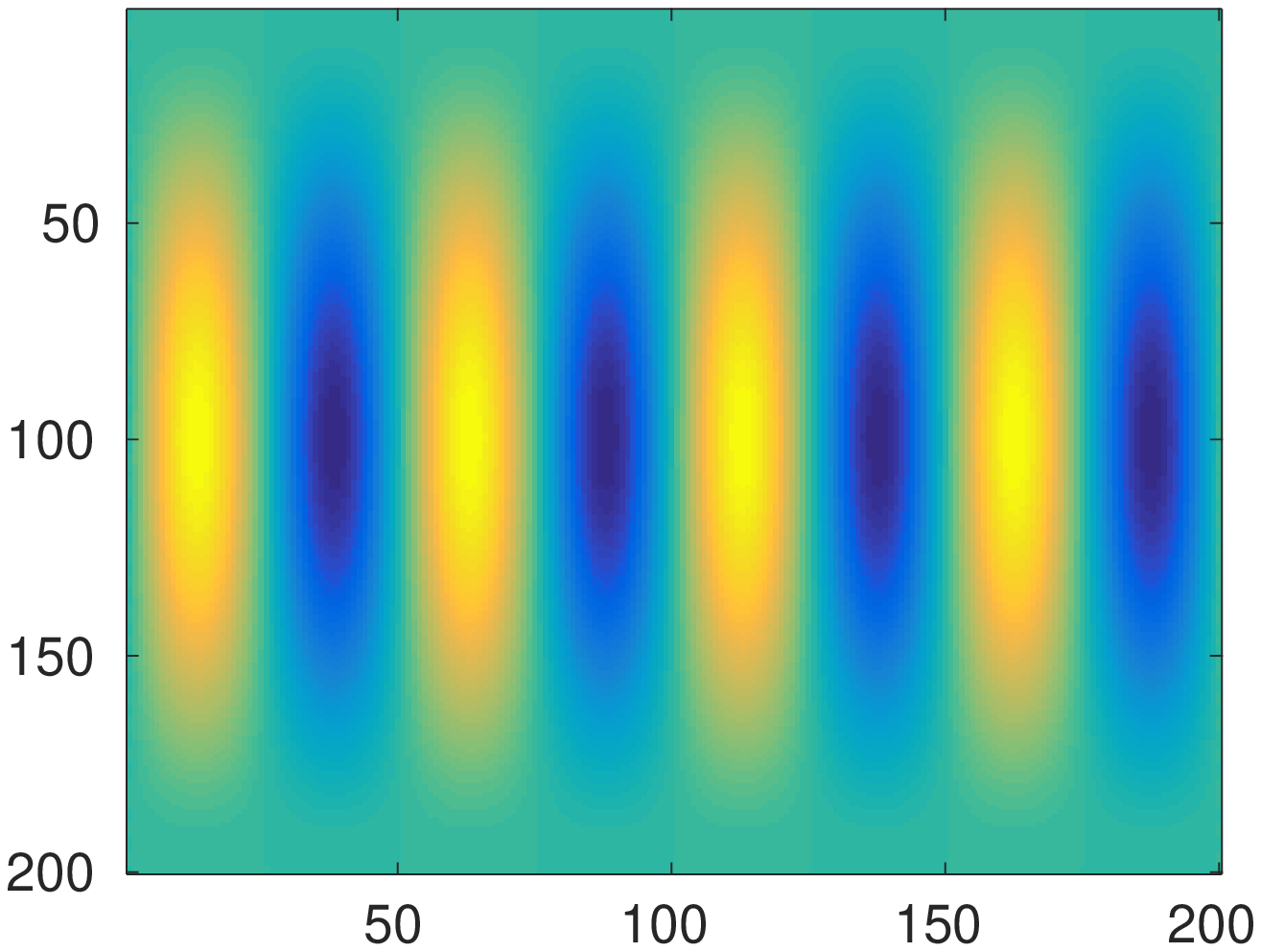}}
 	  \subfigure[C-C reconstruction]{\includegraphics[scale = .28]{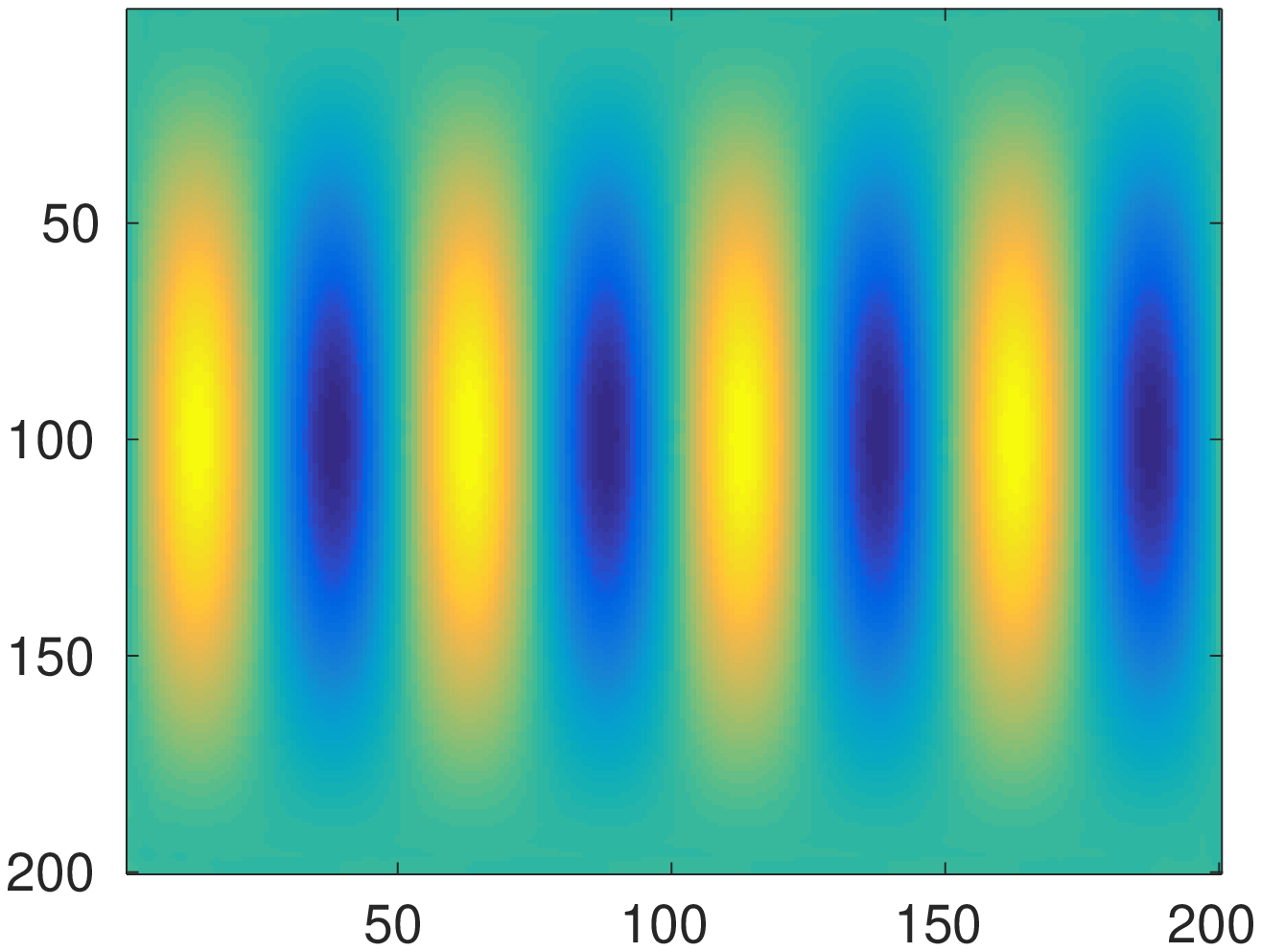}}
        \caption{Reconstruction of $f_2(\vx)$ in Example \ref{ex:example2} with jittered sampling.}
\label{fig:f2image_jit}
\end{figure}

\begin{figure}[htbp]
\centering
\begin{tabular}{c c c}
\medskip
Rosette& Spiral & Polar \\
\includegraphics[scale = .26]{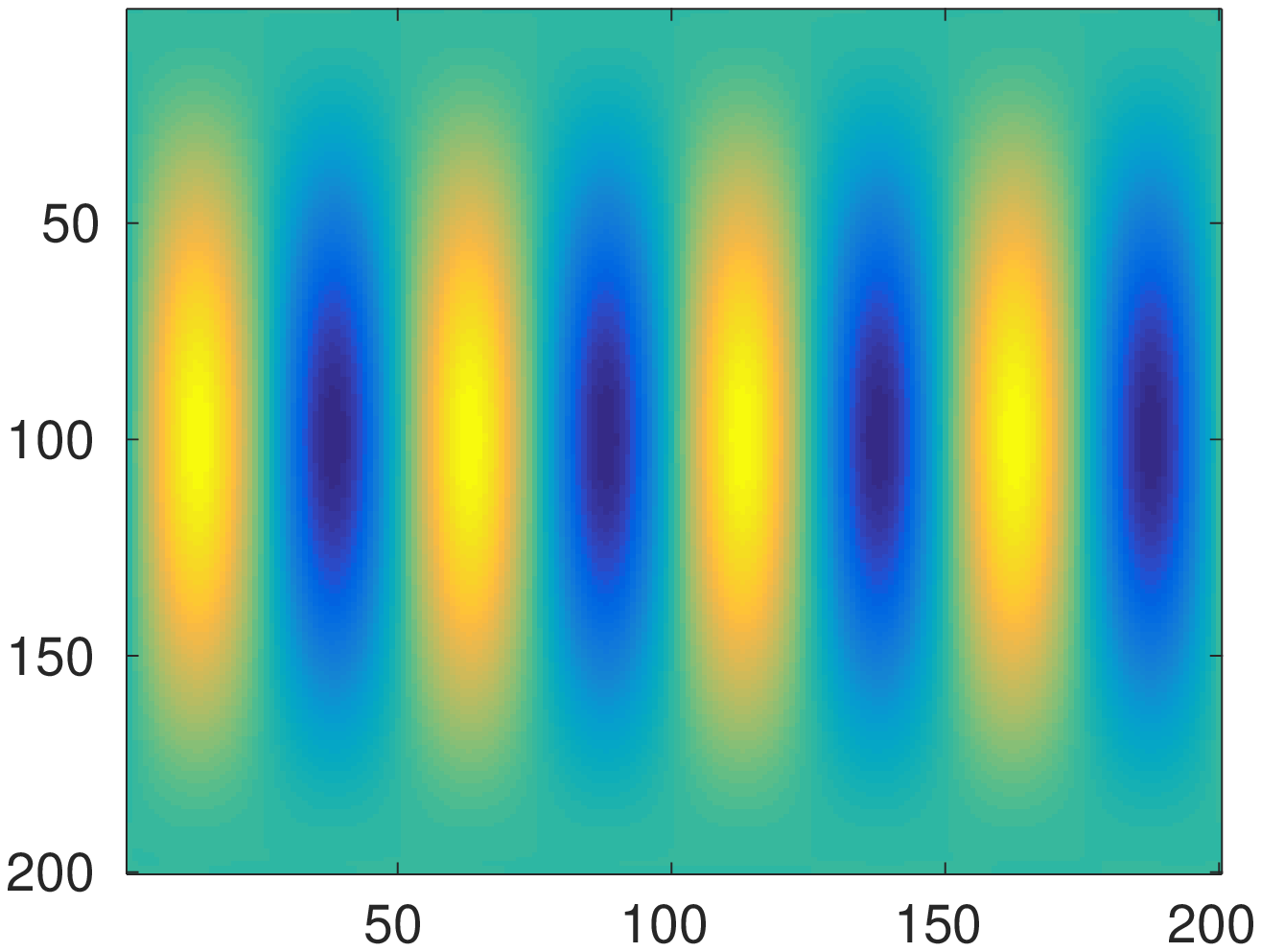} & \includegraphics[scale = .26]{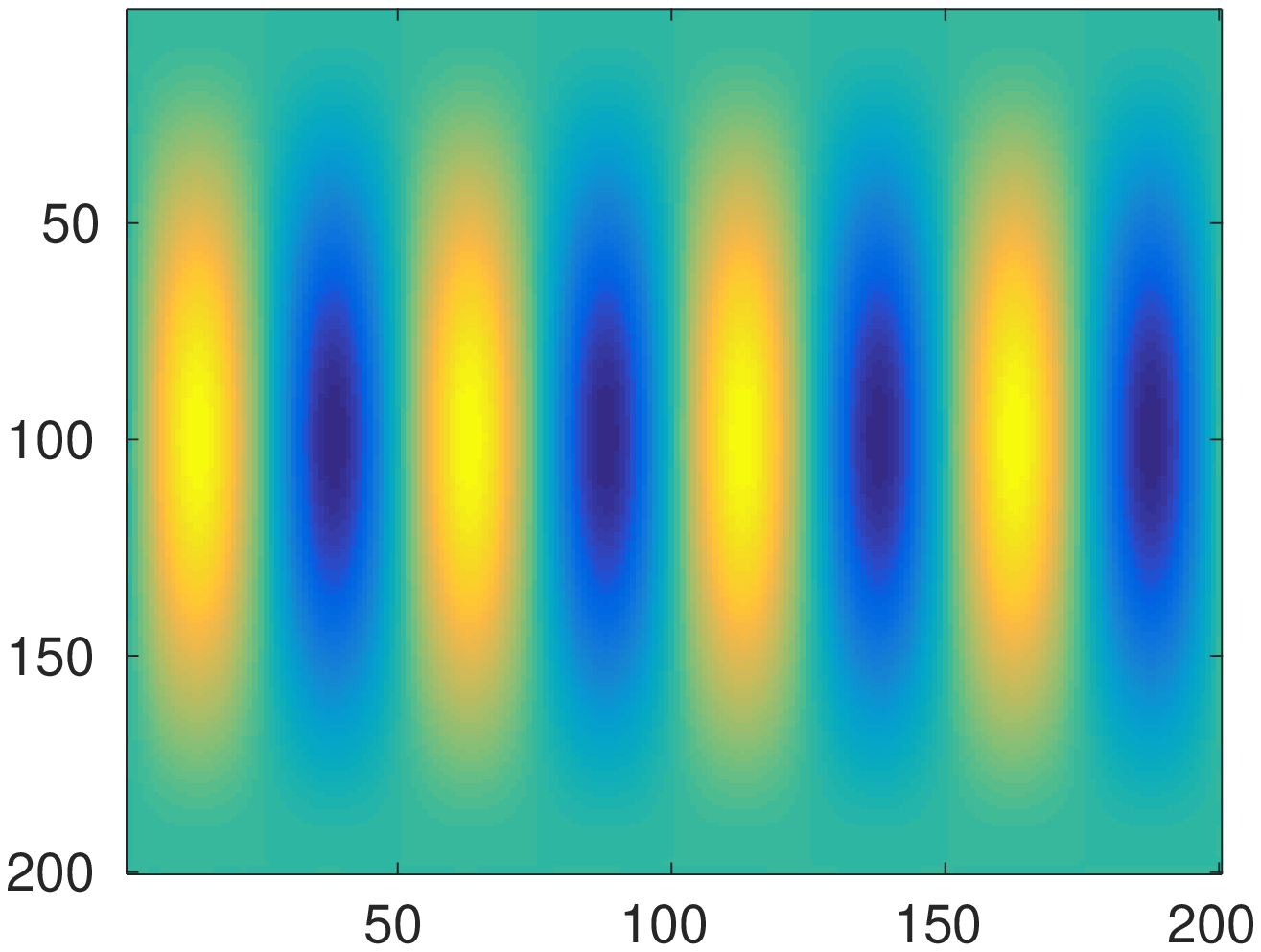} & \includegraphics[scale = .26]{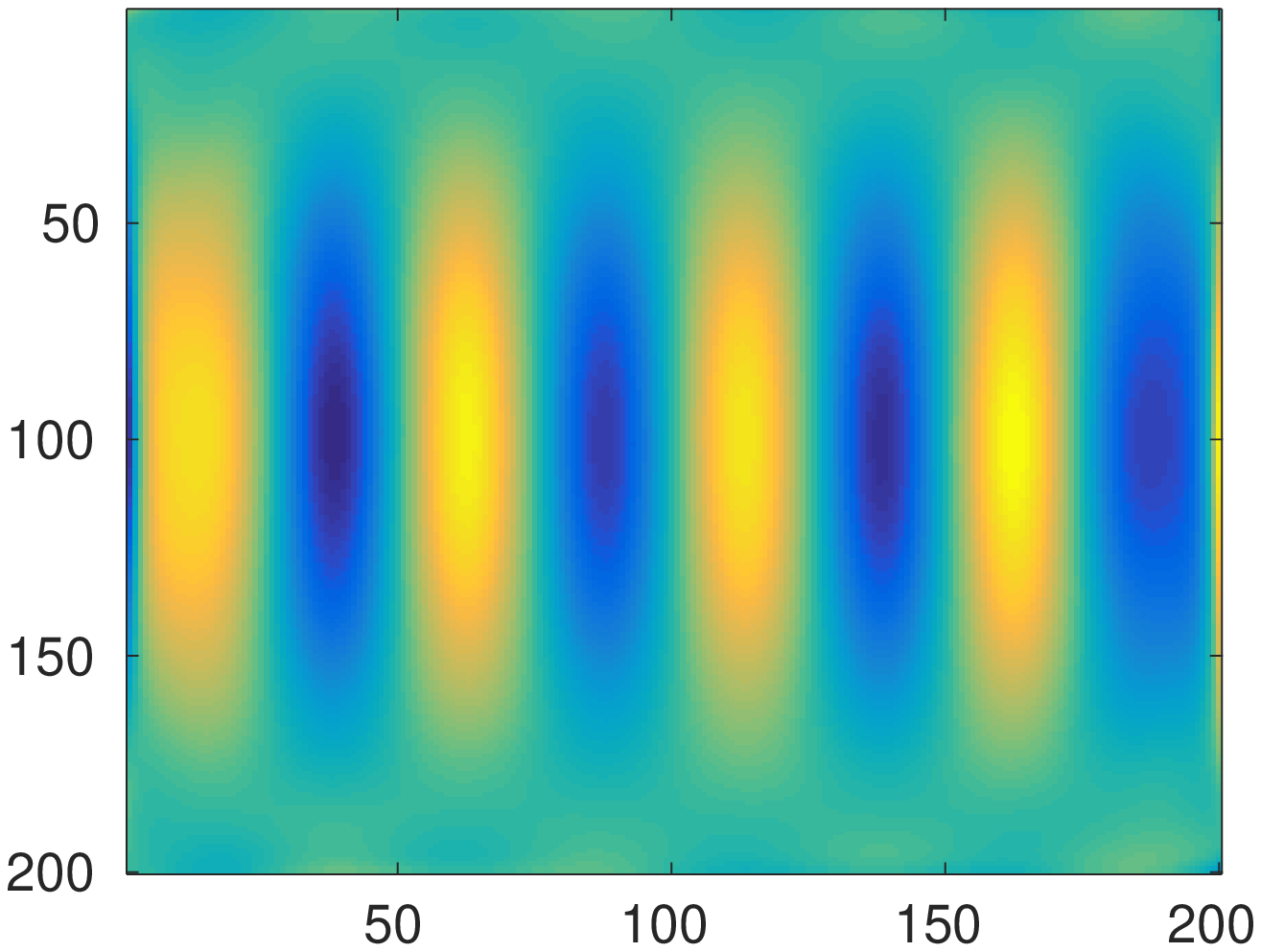}\\
\includegraphics[scale = .26]{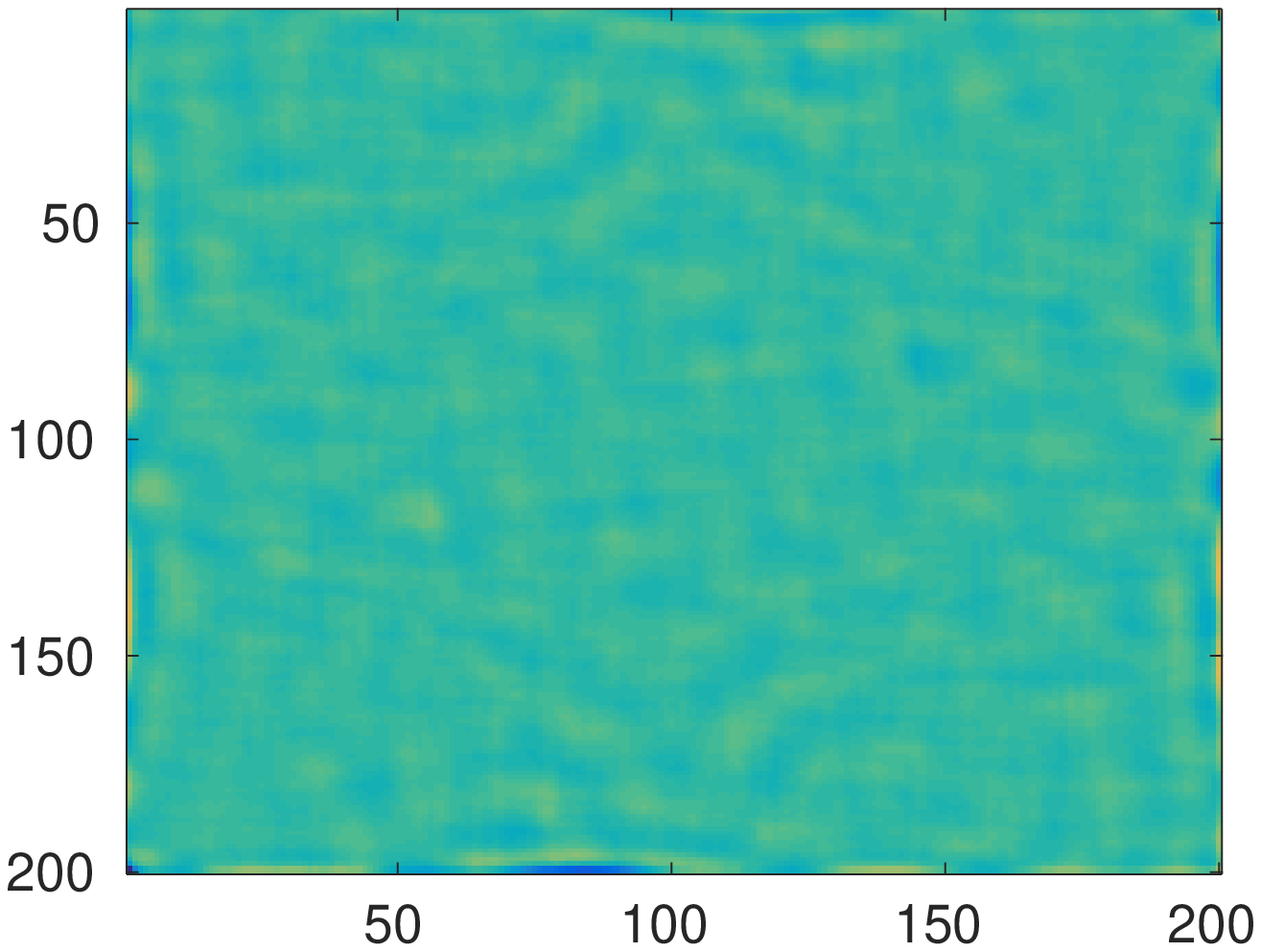} & \includegraphics[scale = .26]{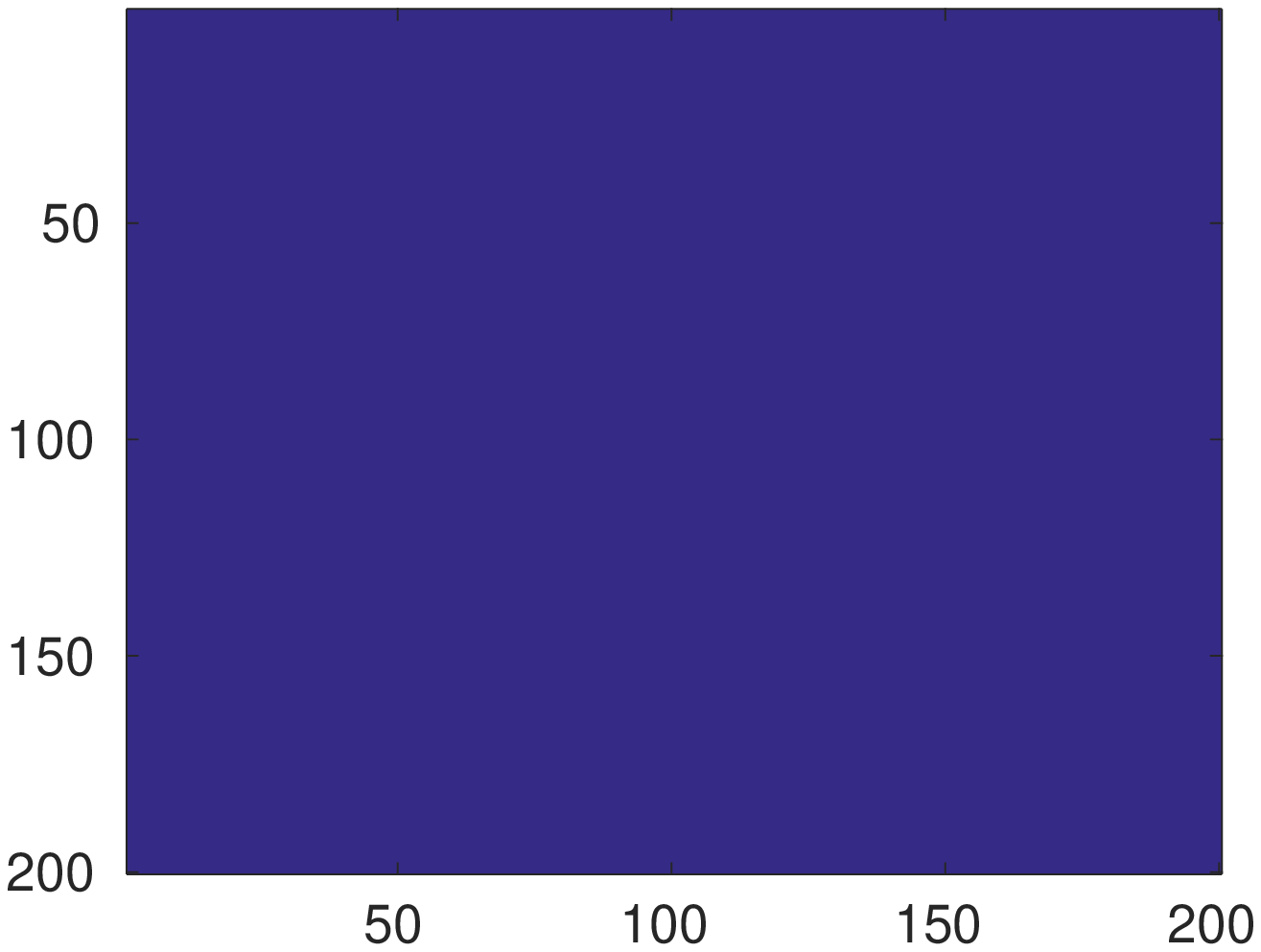} & \includegraphics[scale = .26]{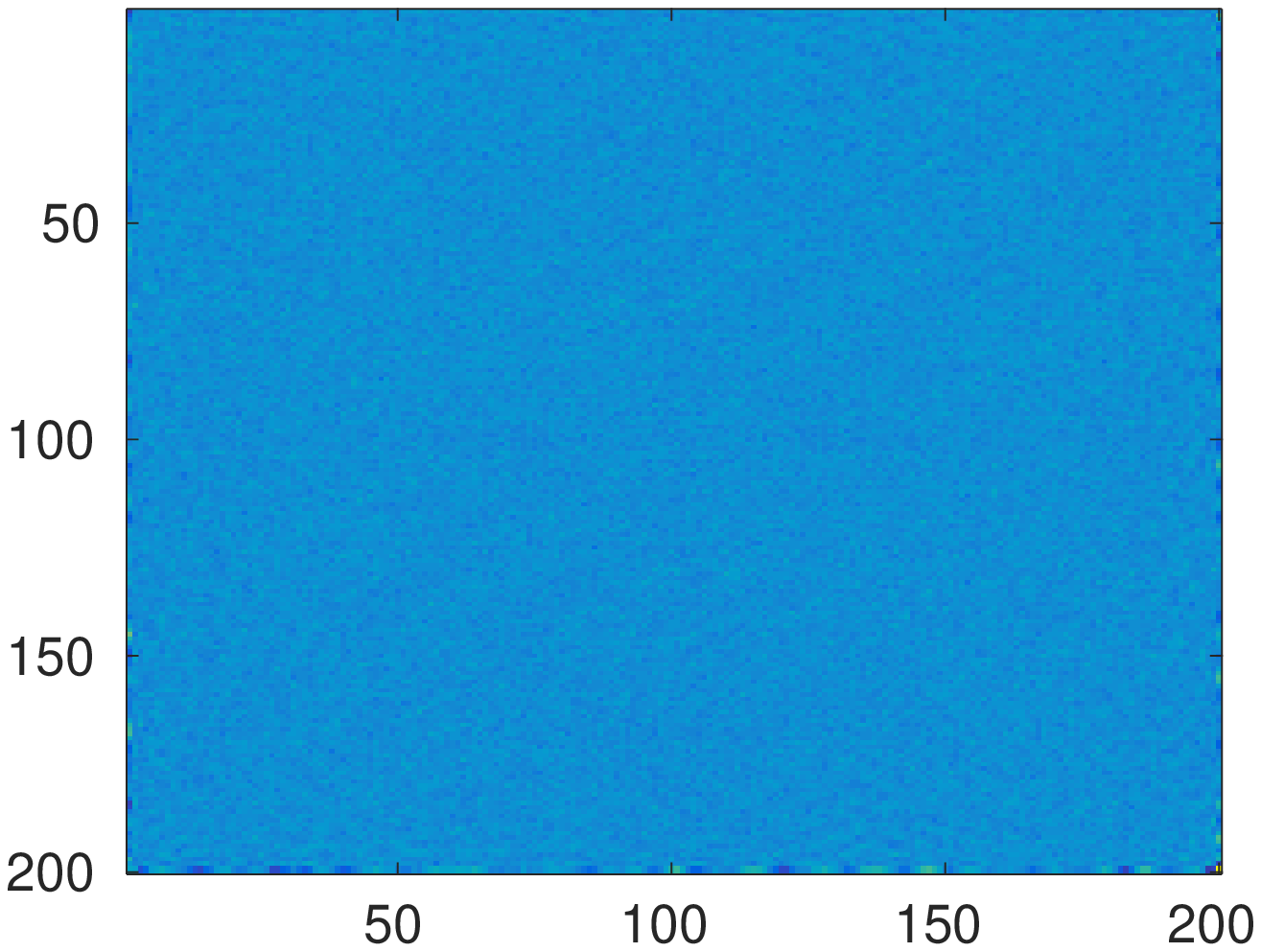}
\end{tabular}
 \caption{Reconstruction of $f_2(\vx)$ in Example \ref{ex:example2} with the A-F (top row) and C-C (bottom row) methods}
\label{fig:f2image_other}
\end{figure}

\begin{figure}[htbp]
\centering
        \subfigure[jittered sampling]{\includegraphics[scale = .15]{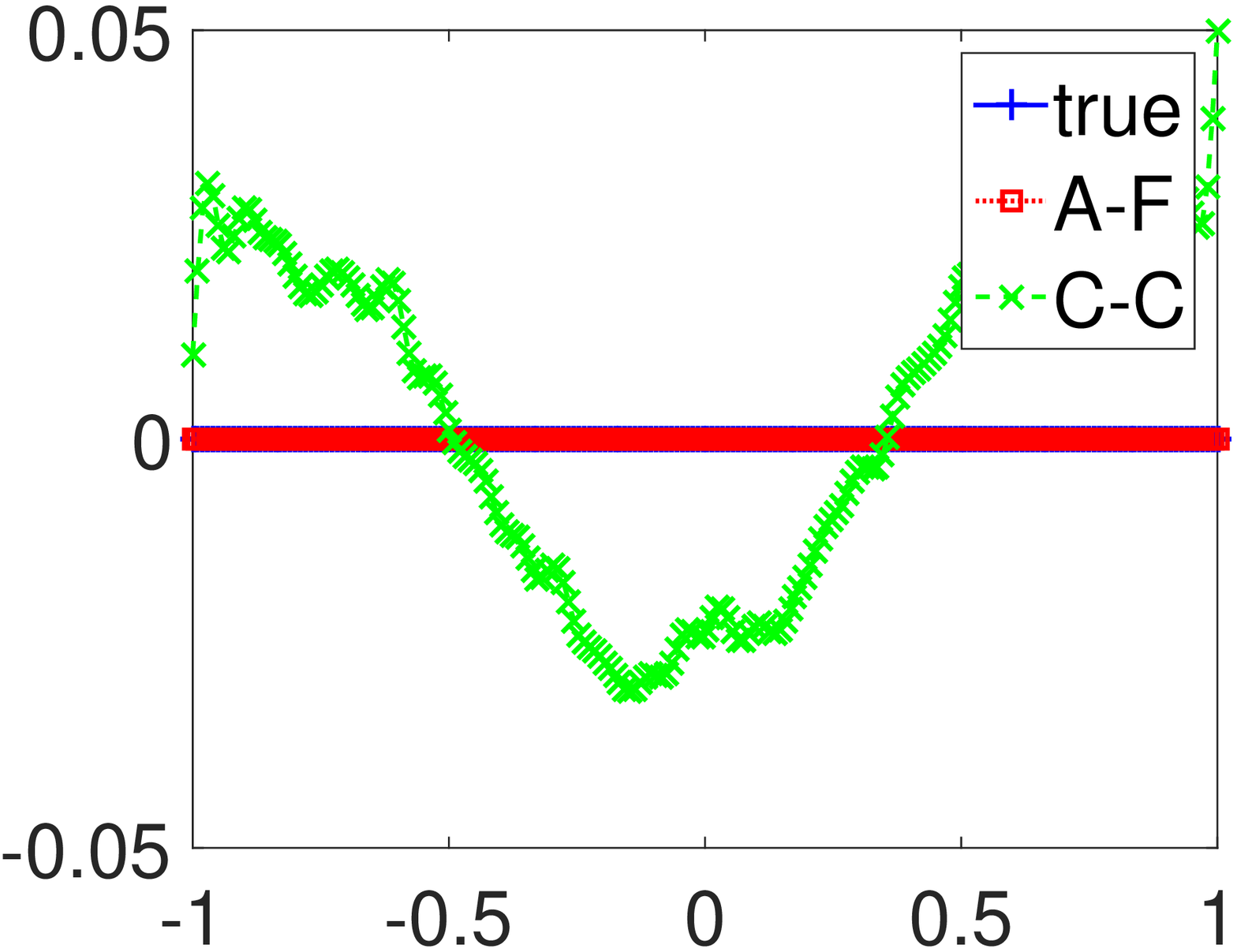}}
        \subfigure[rosette sampling]{\includegraphics[scale = .15]{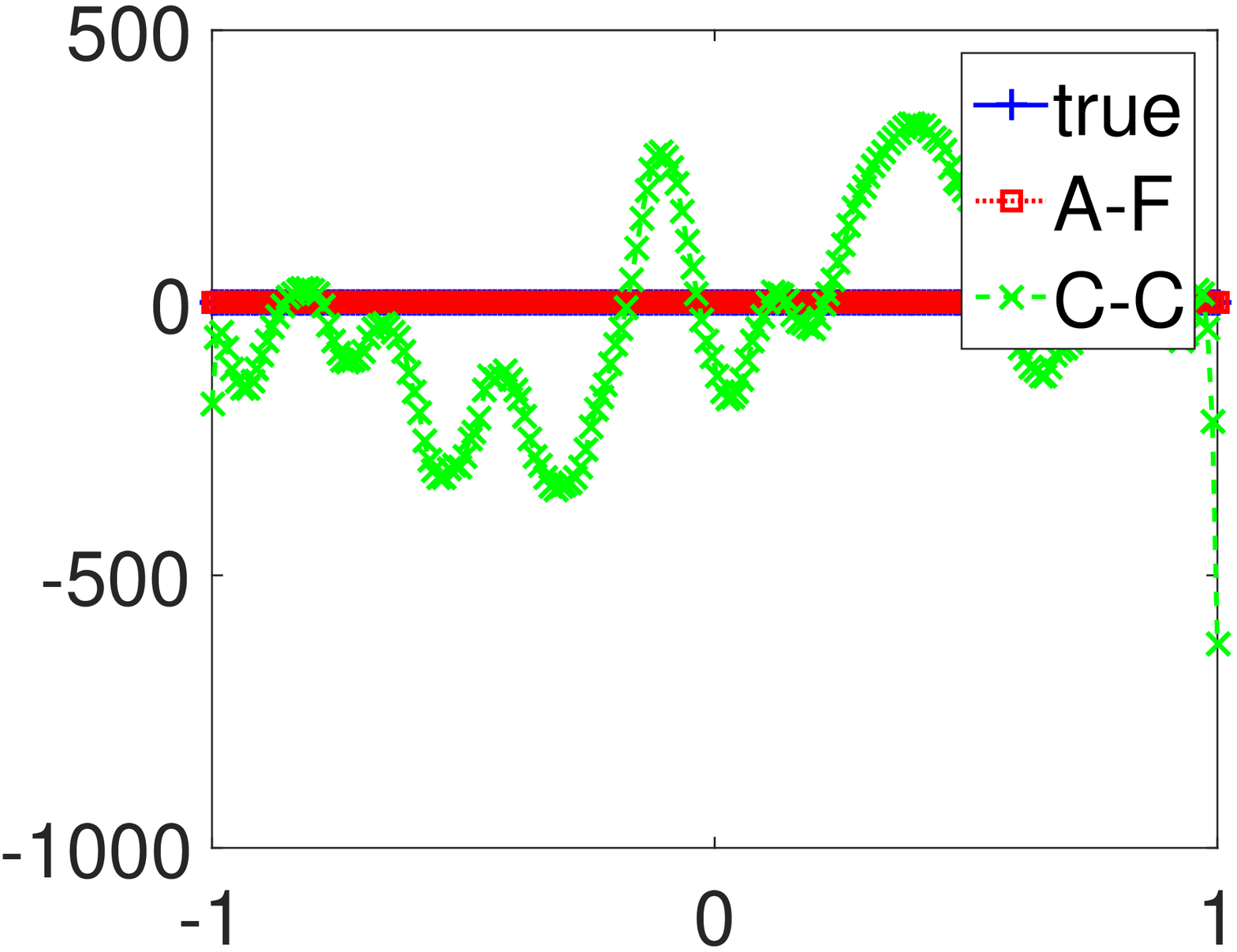}}
        \subfigure[spiral sampling]{\includegraphics[scale = .15]{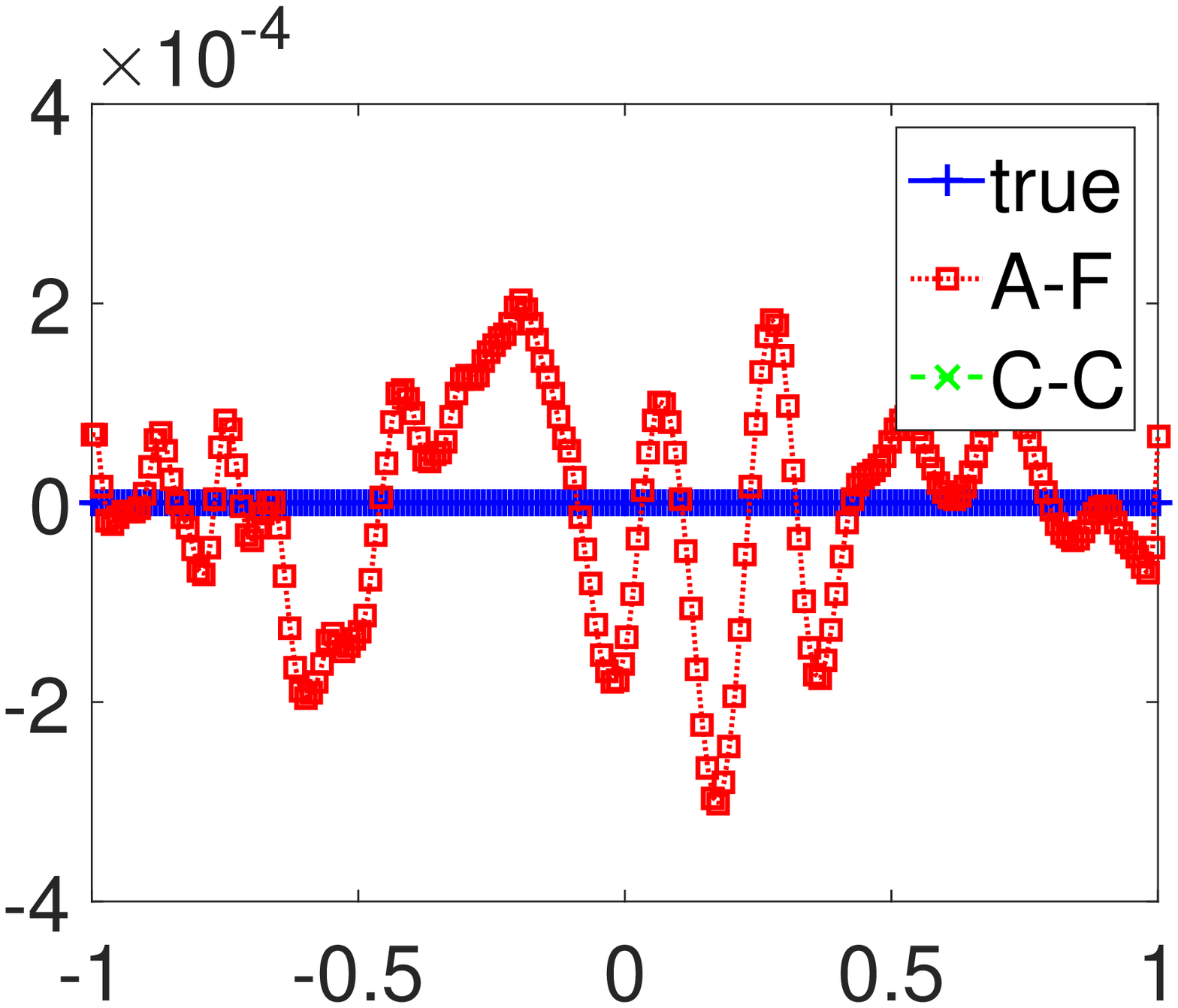}}
        \subfigure[polar sampling]{\includegraphics[scale = .15]{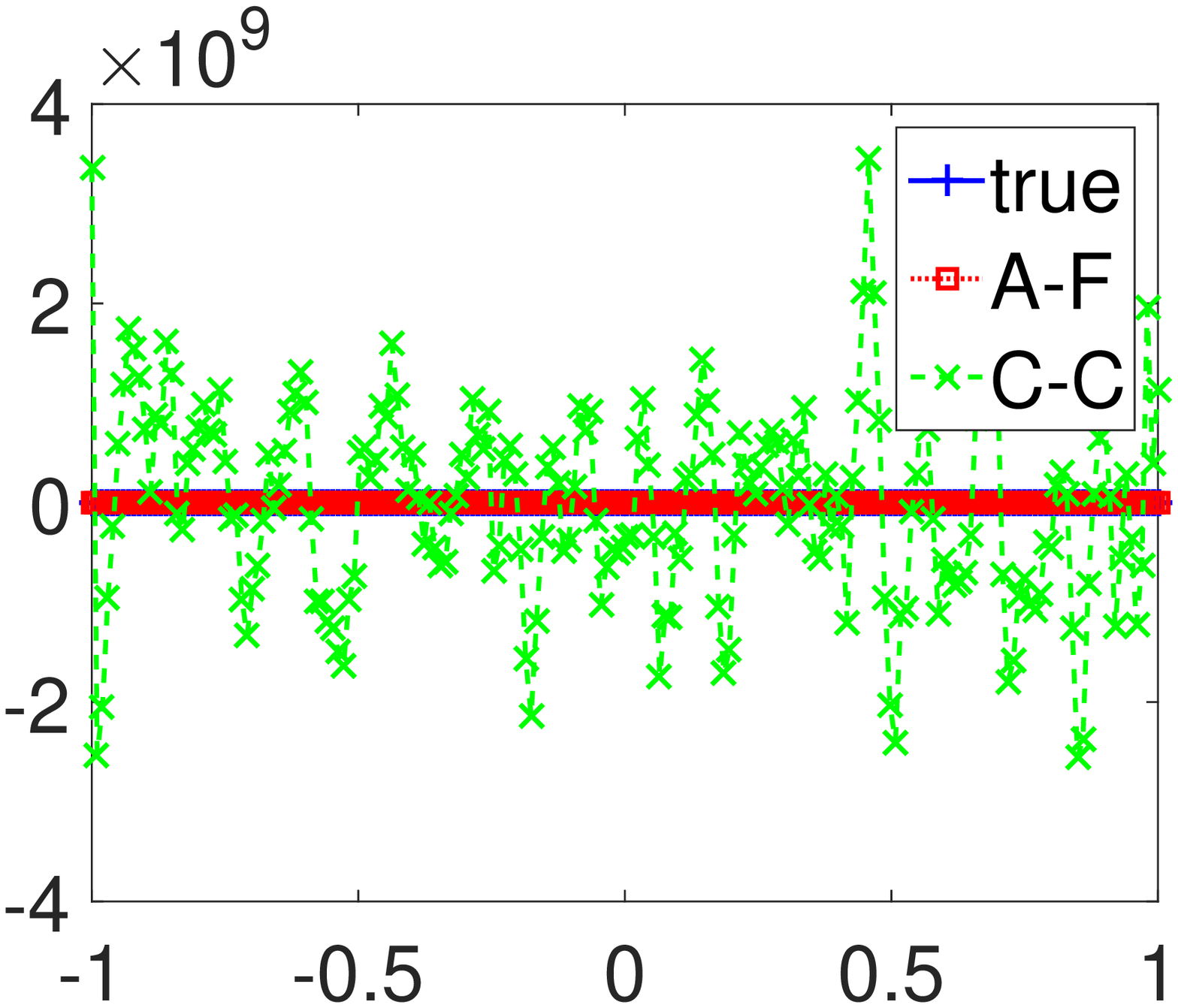}}
        \caption{One-dimensional cross section where $f_2(\vx) = 0$ in Example \ref{ex:example2} for all sampling patterns.}
\label{fig:f2slice}
\end{figure}


\section{Concluding Remarks}
\label{sec:conclusion}
In this paper we established that with some additional constraints, the admissible frame (A-F) method can be extended to two dimensions to accurately approximate the inverse frame operator of a localized frame.  Using our approach, we analyzed the convergence properties of the Casazza Christensen (C-C) method, \cite{Casazza2000, Christensen2000}, for the case when the sampling frame is well localized.  In applications where the given data  do not constitute a well-localized frame, the A-F method can be used to project the corresponding frame data onto a more suitable frame.  As a result, the target function may be approximated as a finite expansion with its asymptotic convergence solely dependent on its smoothness.  This is in contrast to the convergence of the C-C method, which heavily relies on the localization of the sampling frame.

Our numerical experiments demonstrate that when the sampling pattern constitutes a frame,  the A-F method converges more quickly than the C-C method.   Moreover, when the sampling pattern does not constitute a frame, the A-F method still provides a convergent reconstruction, while the C-C method may not.  Thus it appears that it may be possible to relax the admissible frame conditions in the convergence analysis for the A-F method, which will be the subject of future investigations.  Also, as was shown in \cite{GelbSong2014}, in the case of one-dimensional non-uniform Fourier data, the A-F method can be sped up using the NFFT algorithm.  Conversely, the NFFT algorithm demonstrates improved convergence when the A-F method is used to determine its parameters. Thus future investigations will also seek to exploit the relationship between NFFT and the A-F method in two dimensions to build a fast and accurate image reconstruction from non-uniformly sampled Fourier data.  Finally, we anticipate that the A-F method can also be used to identify edges in images, even when the data are heavily undersampled.  Thus our future work will also focus on applying the A-F method in applications where undersampling is common.

\bibliographystyle{siam}
\bibliography{frame}

\end{document}